\documentclass[letterpaper,11pt]{amsart}
\usepackage{amsmath,amssymb,amsthm,pinlabel,tikz, MnSymbol}
\usepackage{verbatim,mathtools}
\usepackage{tikz-cd}
\usepackage{float}
\usepackage{xcolor} 
\newcommand{\nc}{\newcommand}
\nc{\dmo}{\DeclareMathOperator}
\dmo{\ra}{\rightarrow}
\dmo{\N}{\mathbb{N}}
\dmo{\F}{\mathbb{F}}
\dmo{\Z}{\mathbb{Z}}
\dmo{\R}{\mathbb{R}}
\dmo{\C}{\mathcal{C}}
\dmo{\AC}{\mathcal{AC}}
\dmo{\Mod}{Mod}
\dmo{\PMod}{PMod}
\dmo{\B}{B}
\dmo{\PB}{PB}
\dmo{\GL}{GL}
\dmo{\SL}{SL}
\dmo{\Sp}{Sp}
\dmo{\I}{\mathcal{I}}
\dmo{\el}{\ell_{\C}}
\dmo{\NN}{\mathcal{N}}
\dmo{\Tr}{Tr}
\dmo{\rk}{rk}
\dmo{\Aut}{Aut}
\dmo{\Inn}{Inn}
\dmo{\Teich}{Teich}
\dmo{\Ind}{Ind}
\dmo{\cd}{cd}
\dmo{\forget}{Forget}
\dmo{\Homeo}{Homeo}
\dmo{\Out}{Out}
\dmo{\MCG}{MCG}
\dmo{\Diffeo}{Diffeo}
\dmo{\push}{Push}
\dmo{\capd}{Cap}
\dmo{\CG}{CG}
\dmo{\UCG}{UCG}
\dmo{\FCG}{FCG}
\dmo{\PFB}{PFB}
\dmo{\BG}{B}
\dmo{\PBG}{PB}
\dmo{\NRH}{NRH}
\dmo{\st}{st}
\dmo{\lk}{lk}
\nc{\nt}{\newtheorem}
\newtheorem{thm}{{\bf Theorem}}[section]
\newtheorem{mainthm}{Theorem}
\newtheorem{maincor}{Corollary}

\newtheorem{lem}[thm]{{\bf Lemma}}
\newtheorem{cor}[thm]{{\bf Corollary}}
\newtheorem{prop}[thm]{{\bf Proposition}}

\theoremstyle{remark}
\newtheorem{remark}[thm]{Remark}

\newtheorem*{ques}{Question}
\theoremstyle{definition}
\newtheorem{definition}[thm]{Definition}

\numberwithin{equation}{section}

\newtheorem*{rep@theorem}{\rep@title}
\newcommand{\newreptheorem}[2]{%
\newenvironment{rep#1}[1]{%
 \def\rep@title{#2 \ref{##1}}%
 \begin{rep@theorem}}%
 {\end{rep@theorem}}}
\newreptheorem{theorem}{Theorem}

\title[Acylindrical hyperbolicity of Out(RAAG)]
{Acylindrical hyperbolicity of outer automorphism groups of right-angled Artin groups}
\author{Hyungryul Baik}
\address{%
        Department of Mathematical Sciences, KAIST,
		291 Daehak-ro Yuseong-gu, Daejeon, 34141, South Korea
}
\email{%
        hrbaik@kaist.ac.kr
}
\author{Junseok Kim}
\address{%
        Department of Mathematics,
        Technion - Israel Institute of Technology,
        Technion City, Haifa, Israel, 3200003
}
\email{%
        jsk8818@campus.technion.ac.il
}
\begin{document}
\maketitle
\begin{abstract}
 We study the acylindrical hyperbolicity of the outer automorphism group of a right-angled Artin group $A_\Gamma$. When the defining graph $\Gamma$ has no SIL-pair (separating intersection of links), we obtain a necessary and sufficient condition for $\Out(A_\Gamma)$ to be acylindrically hyperbolic. As a corollary, if $\Gamma$ is a random connected graph satisfying a certain probabilistic condition, then $\Out(A_\Gamma)$ is not acylindrically hyperbolic with high probability. When $\Gamma$ has a maximal SIL-pair system, we derive a classification theorem for partial conjugations. Such a classification theorem allows us to show that the acylindrical hyperbolicity of $\Out(A_\Gamma)$ is closely related to the existence of a specific type of partial conjugations.
\end{abstract}
\section{Introduction}
 Given a finite simplicial graph $\Gamma$ whose vertex set and edge set are $V$ and $E$, respectively, the right-angled Artin group (RAAG) $A_\Gamma$ associated with $\Gamma$ is defined by the following group presentation: $$A_\Gamma = \langle v \in V \mid [u, v] = 1 ~\text{for}~\{u, v\} \in E \rangle.$$
In these settings, $\Gamma$ is said to be the defining graph of $A_\Gamma$.

RAAGs can be considered to be a family of groups which interpolate the free groups and free abelian groups. Their outer automorphism groups are of interest of many authors. $\Out(\F_n)$ has been particularly studied a lot. Given that the extended mapping class group $\MCG^{\pm}(S_g)$ of a closed surface $S_g$ is isomorphic to $\Out(\pi_1(S_g))$, $\Out(\F_n)$ naturally share many properties of the mapping class group of punctured surfaces but has a rather rich structure.

Our main focus here is $\Out(A_\Gamma)$. As we mentioned already, two extreme ends of examples of $A_\Gamma$ are free groups $\F_n$ and free abelian group $\Z^n$. $\Out(\F_n)$ and $\GL_n(\Z) (=\Out(\Z^n))$ already share some nice properties. They both are virtually torsion-free, residually finite, and have finite virtual cohomological dimension. Charney-Vogtmann showed in \cite{charney2009finiteness} and \cite{charney2011subgroups} that these properties hold for $\Out(A_\Gamma)$, and the algorithm in \cite{day2021calculating} gave a complete answer to the virtual cohomological dimension of $\Out(A_\Gamma)$.

Recently, there has been significant research activity to understand plenty of group theoretic properties of $\Out(A_\Gamma)$, for example, Kazhdan's property (T) and virtual duality. However, the understanding of when $\Out(A_\Gamma)$ has property (T) or is a duality group is still incomplete. \cite{wade2023note} and \cite{sale2024virtual} provide partial answers to these questions.

It is also true that both $\Out(\F_n)$ and $\GL_n(\Z) $ are non-relatively hyperbolic. The second author with Sangrok Oh and Philippe Tranchida \cite{kim2022automorphism} studied relative hyperbolicity of $\Out(A_\Gamma)$. They showed that unless $\Out(A_\Gamma)$ is either finite or virtually a RAAG $A_\Lambda$ where $\Lambda$ is either a single vertex or disconnected, $\Out(A_\Gamma)$ is never relatively hyperbolic.

But the story becomes more mysterious when it comes to acylindrical hyperbolicity. First of all, $\Out(\F_n)$ is acylindrincally hyperbolic for all $n \geq 2$ but $\GL_n(\Z)$ is not acylindrincally hyperbolic for all $n \geq 3$. Hence it is a curious question how acylindrical hyperbolicity of $\Out(A_\Gamma)$ behaves for different $\Gamma$. Indeed Genevois asked the following question with a complete classification of acylindrical hyperbolicity of $\Aut(A_\Gamma)$ in \cite{genevois2018automorphisms}.

\begin{ques}
    When is $\Out(A_\Gamma)$ acylindrincally hyperbolic? Can we classify such $\Gamma$'s?
\end{ques}

Despite our classification for certain families of graphs, determining the exact status of acylindrical hyperbolicity for general graphs remains an intriguing open challenge. For instance, it would be highly interesting to determine the status of $\Out(A_\Gamma)$ when $\Gamma$ possesses SIL-pairs but fails to admit a maximal SIL-pair system (e.g., the graphs constructed by Wiedmer in \cite{wiedmer2024right}, which will be discussed at the end of Section~\ref{sec:AH}).

Besides from the desire to understand how $\Out(A_\Gamma)$ ``sees" the similarity and difference between $\Out(\F_n)$ and $\GL_n(\Z)$, this is a question of its own interest, since acylindrically hyperbolic groups have many nice properties such as SQ-universality. \cite{osin2016acylindrically} and \cite{hagen2019beyond} would assist in understanding the notion of acylindrical hyperbolicity.

In this paper, we provide a partial answer to this question. First we obtain a complete answer to this question in the case $\Gamma$ does not contain a so-called SIL-pair. A SIL-pair is a pair of non-adjacent vertices $a, b$ which separate the graph in a reasonable sense. Namely, there must be another vertex which is separated from $a, b$ by the intersection of the links of $a, b$(see Definition \ref{def:SIL}). Then the following theorem can be obtained easily from the work of Guirardel-Sale \cite{guirardel2018vastness}.

\begin{mainthm} (Corollary \ref{nosilpair}) \label{mainthmA}
Suppose $\Gamma$ has no SIL-pair. Then $\Out(A_\Gamma)$ is acylindrically hyperbolic if and only if the complement of the star of each vertex is connected and there exist two equivalent vertices $a, b$ such that if $v \leq w$ for some vertices $v,w\in \Gamma$, then $\{v, w\} = \{a, b\}$
\end{mainthm}

For two distinct vertices $v$ and $w$, the partial order $v \leq w$ in Theorem \ref{mainthmA}, called \textit{link-star order}, means $\lk(v) \subset \st(w)$, and $a$ and $b$ are equivalent if and only if $a \leq b$ and $b \leq a$. See Section~\ref{sec:prelim} for details.
Theorem~\ref{mainthmA} has an interesting consequence for random RAAGs(i.e., RAAGs with random defining graphs). Let $\Gamma(n, p)$ denote the Erd\"os-R\'enyi random graph with $n$ vertices and probability $p$ for each edge, where $p=p(n)$ may depend on $n$. We say $\Gamma(n, p)$ has some property with high probability if the probability for the given property to hold approaches to 1 as $n$ approaches to infinity. For background on random graph theory, we refer the reader to the monographs \cite{bollobas2001random, janson2000random}.

It is worth placing this result in the context of the works by Charney--Farber \cite{charney2012random} and Day \cite{day2012finiteness}, who also studied $\Out(A_\Gamma)$ for random graphs $\Gamma$ in the Erd\"os-R\'enyi model. They showed that for any fixed edge probability $p\in (1-1/\sqrt{2},1)$ (independent of $n$), the link-star partial order on the vertices of $\Gamma(n,p)$ is trivial with high probability (see \cite[Theorem 5.1]{charney2012random}), and consequently $\Out(A_\Gamma)$ is finite with high probability. A finite group is trivially not acylindrically hyperbolic, so Corollary~\ref{maincor1} below is already known in the regime where $p$ is a fixed constant in $(1-1/\sqrt{2},1)$. The significance of Corollary~\ref{maincor1} is that it extends to the regime where $p=p(n)$ depends on $n$. In particular, when $np(1-p)<\log n+\log\log n$ (e.g., $1-p\sim\log\log n/n$), transvections of infinite order exist with high probability and $\Out(A_\Gamma)$ is therefore infinite with high probability, yet still fails to be acylindrically hyperbolic. The extreme case $qn=O(1)$ (i.e., $\Gamma$ is very close to the complete graph) is not covered by the present statement.
\begin{maincor} \label{maincor1}
Let $\Gamma=\Gamma(n,p)$ and let $q=1-p$. Assume
\[
p \,\geq\, \frac{\log n+2\log\log n+\omega(n)}{n}\quad\text{with}\quad \omega(n)\to\infty,
\qquad\text{and}\qquad qn\to\infty.
\]
Then $\Out(A_\Gamma)$ is not acylindrically hyperbolic with high probability.
\end{maincor}
The proof of Corollary~\ref{maincor1} is given in Section~\ref{sec:proof-main-cor}. Throughout, we write $f(n)=o(g(n))$ to mean $f(n)/g(n)\to 0$ as $n\to\infty$, and $f(n)=O(g(n))$ to mean that there exist constants $C>0$ and $n_0$ such that $|f(n)|\le C|g(n)|$ for all $n\ge n_0$, following the standard asymptotic notation in random graph theory (see, e.g., \cite{bollobas2001random}).

The case $\Gamma$ contains a SIL-pair is much more complicated and a complete classification is out of reach at the moment. But when $\Gamma$ contains a maximal SIL-pair system, we give a full classification of the partial conjugations in Proposition \ref{classofpartialconj}. In some cases, we can understand the structure of the pure symmetric outer automorphism group $\mathrm{PSO}(A_{\Gamma})$ which is the image of the subgroup of $\Aut(A_\Gamma)$ generated by partial conjugations under the quotient map $\Aut(A_\Gamma) \to \Out(A_\Gamma)$. As a result, we obtain the following.

\begin{mainthm} (Theorem \ref{thm:structureofPSO} + Theorem \ref{beingAH}) \label{mainthmB}
Suppose $\Gamma$ has a maximal SIL-pair system with no additional components (such as a component $D_1$ in Figure \ref{examplesofdecompositions} or components $D_1, D_2$ in Figure \ref{twoadditionalcomponents}).
\begin{enumerate}
\item If there are no nontrivial subordinate partial conjugations (type (3) partial conjugations from the classification in Proposition \ref{classofpartialconj}), then $\mathrm{PSO}(A_{\Gamma})$ is a direct product of two subgroups. In particular, if there are no nontrivial subordinate partial conjugations and no transvections, then $\Out(A_\Gamma)$ is not acylindrically hyperbolic.
\item
Suppose every vertex in simultaneously shared components (such as components $C_i$'s in Figure \ref{examplesofdecompositions} and Figure \ref{twoadditionalcomponents}) either has the star whose complement is connected or defines a subordinate partial conjugation. Then under some minor assumption (this is conjectured to be always satisfied. See Corollary \ref{beingAHcor} and a discussion before it.), $\mathrm{PSO}(A_\Gamma)$ is acylindrically hyperbolic.
\end{enumerate}
\end{mainthm}

Theorem~\ref{mainthmB} shows that when there exists a maximal SIL-pair system with no additional components, the acylindrical hyperbolicity of $\mathrm{PSO}(A_\Gamma)$ is closely related to the existence of subordinate partial conjugations (the two cases in the theorem take care of two extreme situations). We hope to generalize Theorem \ref{mainthmB} for the cases between two cases in the theorem in the future work.

One importance of Theorem~\ref{mainthmB} lies in the fact that understanding the acylindrical hyperbolicity of $\mathrm{PSO}(A_\Gamma)$ is indeed related to that of $\Out(A_{\Gamma})$. In particular, if there is no nontrivial link-star order between two vertices in $\Gamma$ and if $\mathrm{PSO}(A_{\Gamma})$ is not acylindrically hyperbolic, then $\Out(A_\Gamma)$ is not acylindrically hyperbolic as well. However, it is not known whether the acylindrical hyperbolicity of $\mathrm{PSO}(A_\Gamma)$ implies that of $\Out(A_\Gamma)$, even in the case of having finite index subgroup relation. In general, it is conjectured that acylindrical hyperbolicity is a quasi-isometry invariant.

The case with additional components is indicated in Theorem \ref{thm:withaddcpnts}. In this case, if there are $m$ additional components, a finite index subgroup $\Out^*(A_\Gamma)$ of $\Out(A_\Gamma)$ has a normal subgroup, which is the direct product of $m$ copies of a RAAG $A_{\Lambda}$. This also implies non-acylindrical-hyperbolicity of $\Out(A_\Gamma)$ when there are at least two additional components.

Theorem~\ref{mainthmB} can be used to obtain a very explicit infinite family of graphs $\Gamma$ such that $\mathrm{PSO}(A_{\Gamma})$ is acylindrically hyperbolic. To illustrate the examples, for $m\geq 3$, consider a $3m-$gon with vertices labelled $v_1,\ldots,v_{3m}$ and two green vertices outside this polygon. We add edges to connect the vertices labelled $v_{3k}$ (the black vertices in the figure) to the two green vertices, the vertices labelled $v_{3k+1}$ (the red vertices in the figure) to one of the two green vertices, and the vertices labelled $v_{3k+2}$ (the blue vertices in the figure) to the other. The graph obtained in this way is denoted by $\Lambda_m$. For example, the graph on the right in Figure \ref{Fig:Lambda} represents $\Lambda_3$. For $m=2$, the graph $\Lambda_2$ is given on the left in Figure \ref{Fig:Lambda}.

Suppose $p,q,r\geq 2$ are integers. Let $\Gamma(p,q,r)$ be a graph obtained by the following way: consider the disjoint union of three graphs $\Lambda_p$, $\Lambda_q$, and $\Lambda_r$, and identify all the top green vertices as a single point and all the bottom green vertices as a single point. See Figure \ref{Fig:Gamma(2,4,3)} as an example. Now we can apply Theorem \ref{mainthmB} to obtain the following (see Section \ref{subsec:examples} for details).
\begin{maincor} \label{maincor2}
For any integers $p,q,r\geq 2$, $\mathrm{PSO}(A_{\Gamma(p,q,r)})$ is an acylindrically hyperbolic finite index subgroup of $\Out(A_{\Gamma(p,q,r)})$.
\end{maincor}
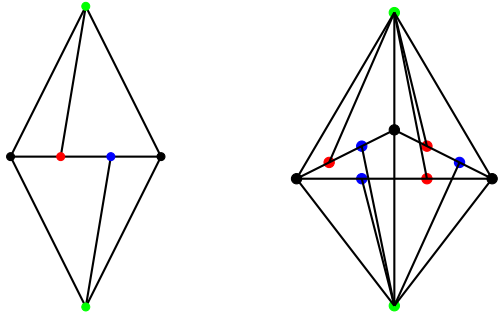
\begin{figure}[ht]
\centering
\begin{tikzpicture}
\draw[black, thick] (1,0) -- (-1,0);
\draw[black, thick] (1,0) -- (0,-2);
\draw[black, thick] (1,0) -- (0,2);
\draw[black, thick] (-1,0) -- (0,-2);
\draw[black, thick] (-1,0) -- (0,2);
\draw[black, thick] (-1/3,0) -- (0,2);
\draw[black, thick] (1/3,0) -- (0,-2);
\filldraw[black] (-1,0) circle (1.5pt);
\filldraw[blue] (1/3,0) circle (1.5pt);
\filldraw[red] (-1/3,0) circle (1.5pt);
\filldraw[black] (1,0) circle (1.5pt);
\filldraw[green] (0,2) circle (1.5pt);
\filldraw[green] (0,-2) circle (1.5pt);
\end{tikzpicture}
\qquad \qquad
\begin{tikzpicture}[scale = 1.3]
\filldraw[green] (0,1.5) circle (1.5pt);
\filldraw[green] (0,-1.5) circle (1.5pt);
\filldraw[black] (0,0.3) circle (1.5pt);
\filldraw[black] (1,-0.2) circle (1.5pt);
\filldraw[black] (-1,-0.2) circle (1.5pt);
\filldraw[red] (1/3,-0.2) circle (1.5pt);
\filldraw[blue] (-1/3,-0.2) circle (1.5pt);
\filldraw[red] (1/3,2/15) circle (1.5pt);
\filldraw[blue] (2/3,-1/30) circle (1.5pt);
\filldraw[red] (-2/3,-1/30) circle (1.5pt);
\filldraw[blue] (-1/3,2/15) circle (1.5pt);
\draw[black, thick] (0,0.3) -- (0,1.5);
\draw[black, thick] (1,-0.2) -- (0,1.5);
\draw[black, thick] (-1,-0.2) -- (0,1.5);
\draw[black, thick] (0,0.3) -- (0,-1.5);
\draw[black, thick] (1,-0.2) -- (0,-1.5);
\draw[black, thick] (-1,-0.2) -- (0,-1.5);
\draw[black, thick] (-1,-0.2) -- (1,-0.2);
\draw[black, thick] (-1,-0.2) -- (0,0.3);
\draw[black, thick] (1,-0.2) -- (0,0.3);
\draw[black, thick] (1/3,-0.2) -- (0,1.5);
\draw[black, thick] (1/3,2/15) -- (0,1.5);
\draw[black, thick] (-2/3,-1/30) -- (0,1.5);
\draw[black, thick] (-1/3,-0.2) -- (0,-1.5);
\draw[black, thick] (2/3,-1/30) -- (0,-1.5);
\draw[black, thick] (-1/3,2/15) -- (0,-1.5);
\end{tikzpicture}
\caption{The graphs $\Lambda_2$ and $\Lambda_3$.}
\label{Fig:Lambda}
\end{figure}

If $\Out(A_{\Gamma(p,q,r)})$ is acylindrically hyperbolic, it gives a partial answer to the Genevois' question. Otherwise, this example is a counterexample to the conjecture that acylindrical hyperbolicity is a quasi-isometry invariant.

The paper is organized as follows: Section \ref{sec:prelim} provides all basic definitions and notations for our discussion. In Section \ref{sec:AH}, we prove Theorem A and study maximal SIL-pair systems of a graph and how it can be used to understand partial conjugations for the corresponding RAAG. In Section \ref{sec:noadditionalcomponents}, we study graphs with no additional components and prove Theorem B.
\subsection{Acknowledgment} We would like to thank Jaehoon Kim for the fruitful conversation which helped the authors to obtain Corollary 1 from Theorem A.

The authors are grateful to Sanghyun Kim, Wonyong Jang, and Sangrok Oh, Richard Wade for the useful discussions and suggestions.
Both authors were supported by the National Research Foundation of Korea (NRF) grant funded by the Korea government (MSIT) (No. 2020R1C1C1A01006912). The second author was also partially supported by Mid-Career Researcher Program (RS-2023-00278510) through the National Research Foundation funded by the government of Korea and by the Israel Science Foundation Grant 1576/23 (PI Nir Lazarovich).

\section{Preliminaries} \label{sec:prelim}
We mostly follow \cite{kim2022automorphism} for the notations. Throughout the paper, $\Gamma$ is always assumed to be a finite simplicial graph with vertex set $V$. We often write $v\in \Gamma$ for a vertex $v$ of $\Gamma$ instead of writing $v \in V$.

For a vertex $v \in \Gamma$, the \textit{link} of $v$, denoted by $\lk(v)$, is the full subgraph of $\Gamma$ spanned by vertices adjacent to $v$.
Similarly, the \textit{star} of $v$, denoted by $\st(v)$, is the full subgraph of $\Gamma$ spanned by vertices adjacent to $v$ and $v$ itself.
For two distinct vertices $v$ and $w$, we then say that $v \leq w$ if $\lk(v) \subset \st(w)$. This partial order induces an equivalence relation on $V$ by setting $v \sim w$ if $v\leq w$ and $w \leq v$.
The partial order then descends to a partial order on the collection of the equivalence classes of vertices by setting $[v] \leq [w]$ if for some, and thus all, representatives $v' \in [v]$ and $w' \in [w]$, we have $v' \leq w'$. A vertex $v \in \Gamma$ is maximal if any vertex $w$ such that $v \leq w$ is equivalent to $v$.

For a subgraph $\Lambda$ of $\Gamma$, the \textit{link} of $\Lambda$, denoted by $\lk(\Lambda)$, is the intersection of the links of all vertices in $\Lambda$. The \textit{star} of $\Lambda$, denoted by $\st(\Lambda)$, is the full subgraph spanned by $\lk(\Lambda)$ and $\Lambda$.

\subsection{Automorphism groups of RAAGs}\label{2.1}

A theorem which was conjectured by Servatius \cite{SERVATIUS} and proved by Laurence \cite{Laurence} says that $\Aut(A_{\Gamma})$ is generated by the following four classes of automorphisms:
\begin{itemize}
    \item \textbf{Graph automorphisms}. An automorphism of $\Gamma$ induces an automorphism of $A_{\Gamma}$ since it preserves the edges of $\Gamma$ and thus the relations of $A_{\Gamma}$. The automorphism obtained this way is called a \textit{graph automorphism}.
    \item \textbf{Inversions}. An automorphism of $A_{\Gamma}$ by sending one generator $v$ to its inverse $v^{-1}$ is called an \textit{inversion}.
    \item \textbf{Transvections}. Take two vertices $v$ and $w$ in $\Gamma$ such that $v \leq w$. Then the automorphism sending $v$ to $vw$ and fixing all the other vertices is called a \textit{right transvection} and is denoted by $R_
    {vw}$. We can similarly define a \textit{left transvection} $L_{vw}$ by sending $v$ to $wv$ and still fixing all the other vertices.
    \item \textbf{Partial conjugations}. For a vertex $v\in\Gamma$, let $C$ be a connected component of $\Gamma - \st(v)$. The automorphism defined by conjugating every vertex in $C$ by $v$ is called a \textit{partial conjugation defined by $v$} and is denoted by $P_v^C$. $C$ is called the \textit{support} of the partial conjugation $P_v^C$. If a component $C$ is of $\Gamma - \st(v)$ is composed of a single vertex $w$, we write $P_v^w$ instead of $P_v^{\{w\}}$.
\end{itemize}
A graph has no transvection if $v\not\leq w$ for any distinct pair of vertices $v$ and $w$ in the graph.

Note that inversions and graph automorphisms have finite order, but transvections and partial conjugations have infinite order.
Let $\Aut^*(A_{\Gamma})$ be the subgroup of $\Aut(A_{\Gamma})$ generated only by transvections and partial conjugations. Then we can easily deduce the following:

\begin{lem}
$\Aut^*(A_{\Gamma})$ is a finite index normal subgroup of $\Aut(A_{\Gamma})$.
\end{lem}

Under the quotient map $\Aut(A_\Gamma)\rightarrow \Out(A_\Gamma):=\Aut(A_\Gamma)/\Inn(A_\Gamma)$, $\Out(A_\Gamma)$ is generated by the images of graph automorphisms, inversions, transvections, and partial conjugations. Similarly, $\Out^*(A_{\Gamma})$ is defined to be the subgroup generated by the images of transvections and partial conjugations in $\Out(A_\Gamma)$; it can be considered as the image of $\Aut^*(A_{\Gamma})$ in $\Out(A_{\Gamma})$. By the above lemma, $\Out^*(A_{\Gamma})$ is also a finite index normal subgroup of $\Out(A_{\Gamma})$. For $f\in \Aut(A_\Gamma)$, denote $[f]\in\Out(A_\Gamma)$ the image of $f$ via the quotient map.

Similarly, we define $\Aut^0(A_{\Gamma})$ and $\Out^0(A_{\Gamma})$ by subgroups generated by transvections, partial conjugations, and inversions. Since $[v]$ and $\st [v]$ are induced subgraphs of $\Gamma$, we can define $A_{[v]}$ and $A_{\st [v]}$ as subgroups of $A_{\Gamma}$. This leads to the following proposition when $\Gamma$ is connected.

\begin{prop}\cite[Proposition 3.2]{charney2009finiteness}
    Let $\Gamma$ be a connected graph and $[v]$ be a maximal equivalence class. Then any $\phi\in \Out^0(A_{\Gamma})$ has a representative $\phi_v \in \Aut^0(A_{\Gamma})$ which preserves both $A_{[v]}$ and $A_{\st [v]}$.
\end{prop}

For each equivalence class $[v]$ of maximal vertex $v$, there is a group homomorphism \[ P_{[v]}: \Out^0(A_\Gamma) \to \Out^0(A_{\lk [v]})\]
defined by the composition of two maps $R_{[v]}:\Out^0 (A_{\Gamma})\to \Out^0(A_{\st [v]})$ and $E_{[v]}:\Out^0 (A_{\st [v]})\to \Out^0 (A_{\lk [v]})$. The former map, called a \textit{restriction map}, sends $\phi$ to the equivalence class of the restriction of $\phi_v$ to $A_{\st [v]}$ obtained by the previous proposition, and the latter map, called an \textit{exclusion map}, sends each generator in $[v]$ to the identity. If $\Gamma$ is a complete graph, we define $P_{[v]}$ to be the trivial map, as $\lk [v]$ is empty.
We can obtain the following group homomorphism obtained by amalgamating all of the projection maps, called the \textit{projection homomorphism}:

\[ P:= \prod P_{[v]} : \Out^0(A_\Gamma) \to \prod \Out^0(A_{\lk [v]}).\]

Charney and Vogtmann showed in \cite{charney2009finiteness} and \cite{charney2011subgroups} that the kernel $\ker P$ of the projection homomorphism has an intriguing subgroup structure, and that its generators can be described in precise terms. To describe them, we recall some definitions. A vertex $v$ is called \textit{leaf-like} if there is a unique maximal vertex $w$ in $\lk(v)$ and this vertex satisfies $[v]\leq [w]$. The transvection $R_{vw}$ is called a \textit{leaf transvection}.

For a connected graph $\Gamma$ and a vertex $v$ in $\Gamma$, we define $x$ and $y$ to be in the same \textit{$\hat{v}$-component} of $\Gamma$ if they can be connected by an edge-path which contains no edges of $\st(v)$. A $\hat{v}$-component lying entirely inside $\st (v)$ is called a \textit{trivial} $\hat{v}$-component. Any other $\hat{v}$-component is called a \textit{nontrivial} $\hat{v}$-component. An automorphism of $\Aut(A_{\Gamma})$ which conjugates all vertices in a single nontrivial $\hat{v}$-component of $\Gamma$ by $v$ is called a \textit{$\hat{v}$-component conjugation}.

\begin{thm}\cite[Theorem 4.2]{charney2009finiteness}\label{kerofprojhom}
For a connected graph $\Gamma$, the kernel $\ker P$ of the projection homomorphism is a free-abelian group generated by the set of all leaf transvections and nontrivial $\hat{v}$-component conjugations for all $v\in \Gamma$.
\end{thm}

\subsection{Relations in $\Out(A_{\Gamma})$}

The following lemma provides a classification of connected components in the complements of stars of two nonadjacent vertices.

\begin{lem}\cite[Lemma 2.1]{day2018subspace}
Let $a$ and $b$ be nonadjacent vertices of $\Gamma$. Then the components in $\Gamma- \st(a)$ consist of $A_0,\cdots,A_k,C_1,\cdots,C_l$ and the components of $\Gamma-\st(b)$ consist of $B_0,\cdots,B_m,C_1,\cdots,C_l$ where $b\in A_0$ and $a\in B_0$, and $A_1,\cdots,A_k \subset B_0$ and $B_1,\cdots,B_m \subset A_0$.
\end{lem}

In the above lemma, $A_0$ and $B_0$ are called the \emph{dominating components}, $C_i$'s are called the \emph{shared components}, and the other components are called the \emph{subordinate components}. Note that any of $k$, $m$, or $l$ can be zero.

\begin{definition}\label{def:SIL}
 Let $\Gamma$ be a graph. We say that a pair of vertices of $\Gamma$, $(a,b)$ form a separating intersection of links (SIL, for short) or SIL-pair if they are not adjacent, and there exists a connected component $C$ of $\Gamma - \left ( \lk(a) \cap \lk(b) \right)$ such that $C$ does not contain $a$ and $b$.
\end{definition}

It is clear that the component $C$ in the definition of SIL-pair is a shared component. Conversely, a shared component must satisfy the definition of SIL-pair. This observation leads to the following lemma.

\begin{lem}\cite[Lemma 2.3]{day2018subspace}\label{lem:SharedComponents}
A pair $(a,b)$ is a SIL-pair if and only if the set
of shared components associated to $(a,b)$ is nonempty.
\end{lem}

We also say that $C$ is a shared component of $(a,b)$.
We say that a graph $\Gamma$ has no SIL if any pair of two vertices of $\Gamma$ is not a SIL-pair.
 
The following lemma is the key to determining whether two partial conjugations commute or not.

\begin{lem}\cite[Lemma 2.4]{day2018subspace}\label{lem:commutingrelation}
Let $a$ and $b$ be nonadjacent vertices in $\Gamma$ such that there are non-trivial partial conjugations $P_a^C$ and $P_b^D$ in $\Out(A_\Gamma)$. Then $[P_a^C, P_b^D]\neq 1$ in $\Out(A_\Gamma)$ if and only if $(a,b)$ is a SIL-pair and one of the following conditions hold:
\begin{itemize}
\item $C$ and $D$ are the dominating components for the pair $(a, b)$.
\item One of $C$ or $D$ is dominating and the other is shared.
\item $C$ and $D$ are identical shared components.
\end{itemize}
\end{lem}

\subsection{Pure symmetric automorphism groups}
The subgroup $\mathrm{PSA}(A_\Gamma)\leq\Aut(A_\Gamma)$ generated by partial conjugations is called the \emph{pure symmetric automorphism group}, and its image via the quotient map $\Aut(A_\Gamma)\to \Out(A_\Gamma)$ is called the \emph{pure symmetric outer automorphism group}, which is denoted by $\mathrm{PSO}(A_\Gamma)$. Toinet originally gave a finite presentation of $\mathrm{PSA}(A_{\Gamma})$ in \cite{toinet2012finitely}. By taking the quotient by the inner automorphism group, we obtain the following presentation of $\mathrm{PSO}(A_{\Gamma})$.

\begin{thm}[\cite{toinet2012finitely}, \cite{day2018subspace}]\label{poutpresentation}
For each graph $\Gamma$, $\mathrm{PSO}(A_{\Gamma})$ has a finite presentation consisting of the standard generating set and relations of the form:
\begin{enumerate}
    \item $[P_{v}^C,P_{w}^D]=1$ when $[v,w]=1$
    \item $[P_{v}^C,P_{w}^D]=1$ when $C\cap D=\emptyset$, $w\notin C$, and $v\notin D$
    \item $[P_{v}^C,P_{w}^D]=1$ when $\{v\}\cup C\subset D$ or $\{w\}\cup D\subset C$
    \item $[P_{v}^C P_{v}^{D},P_{w}^D]=1$ when $w\in C$ and $v\notin D$
    \item $\prod_{C\in I_v}P_{v}^C=1$ where the product is taken over the set $I_v$ of connected components of $\Gamma-\st(v)$.
\end{enumerate}
\end{thm}

Note that $\mathrm{PSO}(A_{\Gamma})=\Out^*(A_{\Gamma})$ when there is no transvection in $\Gamma$. From this observation, it is easy to deduce the following lemma.

\begin{lem}
    Let $\Gamma$ be a graph. Then $\mathrm{PSO}(A_{\Gamma})$ is a finite index subgroup of $\Out^*(A_{\Gamma})$ if and only if there is no transvection in $\Gamma$.
\end{lem}

\subsection{Acylindrically hyperbolic groups}

We refer to the definition of acylindrically hyperbolic groups based on Osin's paper \cite{osin2016acylindrically}.

\begin{definition}
Let $G$ be a group acting on a hyperbolic space $S$ isometrically. An action of $G$ on $S$ is \textit{acylindrical} if for any $\varepsilon>0$ there exist $R,N>0$ such that for each $x,y \in S$ with $\mathrm{d}(x,y)\geq R$, we have
\[ \#\{g\in G : \mathrm{d}(x,gx)\leq \varepsilon, \mathrm{d}(y,gy)\leq \varepsilon\} \leq N.\]
An action of $G$ on $S$ is \textit{non-elementary} if $G$ admits two independent hyperbolic isometries of $S$.
\end{definition}

\begin{definition}
A group $G$ is called acylindrically hyperbolic if $G$ admits a non-elementary acylindrical action on some hyperbolic space $S$.
\end{definition}

It is a well-known fact that non-virtually-cyclic hyperbolic groups and relatively hyperbolic groups with proper peripheral subgroups are acylindrically hyperbolic. The same is true of $\Out(\mathbb{F}_n)$ with $n\geq 2$, $\MCG(S_{g,p})$ unless $g=0$ and $p\leq 3$, and directly indecomposable non-cyclic RAAGs. However, $\SL_n(\Z)$, $\GL_n(\Z)$ with $n\geq 3$ and directly decomposable RAAGs are not acylindrically hyperbolic.

\begin{definition}
 Let $G$ be a group. A subgroup $H < G$ is called \textit{s-normal} if for every $g \in G$, the cardinality $ |H \cap g^{-1}Hg | $ is infinite.
\end{definition}

 The following are well-known facts for acylindrically hyperbolic groups.
 
\begin{prop}\cite[Corollary 1.5]{osin2016acylindrically}\label{normalityAH}
    Let $G$ be an acylindrically hyperbolic group and $H$ be an s-normal subgroup of $G$. Then $H$ is also acylindrically hyperbolic. In particular, if $H$ is an infinite normal subgroup of $G$, then $H$ is acylindrically hyperbolic.
\end{prop}

\begin{prop}\cite[Corollary 7.3 (a)]{osin2016acylindrically}\label{amenable}
    If $G$ is acylindrically hyperbolic, then $G$ contains no infinite amenable normal subgroup. In particular, $G$ does not contain free abelian groups and infinite nilpotent groups as normal subgroups.
\end{prop}

\begin{prop}\cite[Corollary 7.3 (b)]{osin2016acylindrically}\label{product}
    Let $G$ be an acylindrically hyperbolic group and assume that $G$ decomposes as $G=G_1 \times G_2$. Then either $G_1$ or $G_2$ should be finite.
\end{prop}

Proposition~\ref{normalityAH} shows that Baumslag-Solitar group $BS(m,n)=\langle a,b|b^{-1}a^m b = a^n\rangle)$ with nonzero $m,n$ is not acylindrically hyperbolic. This is because the subgroup $\langle a\rangle$ is s-normal subgroup of $BS(m,n)$.

\section{Acylindrical hyperbolicity of $\Out(A_\Gamma)$} \label{sec:AH}

From this section onwards, we use the notation $f \in \Out(A_{\Gamma})$ even though $f$ represents an element in $\Aut (A_{\Gamma})$.

\subsection{$\Gamma$ with no SIL}

We briefly introduce the standard representation of $\Out(A_{\Gamma})$. The abelianization map $A_{\Gamma}\to \Z^n$, where $n$ is the number of vertices of $\Gamma$, induces a representation $\rho : \Out(A_{\Gamma})\to \GL_n(\Z)$. In this setting, the kernel $\mathrm{IA}_{\Gamma}=\ker \rho$ is called the Torelli subgroup of $\Out^*(A_{\Gamma})$.

Day and Wade showed in their papers \cite{day2009symplectic} and \cite{wade2012symmetries} that the Torelli subgroup is generated by all partial conjugations and commutator transvections $[R_{uv},R_{uw}]$. This fact guarantees that $\mathrm{IA}_{\Gamma}$ is a normal subgroup of $\Out^*(A_{\Gamma})$. Furthermore, the kernel of the restriction of $\rho$ to $\Out^*(A_{\Gamma})$ is the same as $\mathrm{IA}_{\Gamma}$. If $\Gamma$ has no SIL, then no commutator transvections exist so that $\mathrm{IA}_{\Gamma}$ is generated by all partial conjugations, which is finitely generated free-abelian by Lemma \ref{lem:commutingrelation}.

In addition, the representation $\rho$ can be modified by reordering the vertices to obtain a group of block-lower triangular matrices as the image of $\rho$. This gives the following exact sequence. See \cite{{guirardel2018vastness}} for details.

\begin{prop}\cite[Proposition 4.1]{guirardel2018vastness}
    Suppose $\Gamma$ has no SIL. Then there is a short exact sequence
    \[
       1 \to P \to \Out^*(A_{\Gamma}) \to \prod_{i=1}^k\SL_{n_i}(\Z) \to 1
    \]
    where $P$ is finitely generated nilpotent, and $n_1,\ldots,n_k$ are the sizes of the equivalence classes in $\Gamma.$
\end{prop}

The above proposition allows us to classify acylindrically hyperbolic $\Out(A_{\Gamma})$ when $\Gamma$ has no SIL with using algebraic properties of acylindrically hyperbolic groups. The definition of acylindrically hyperbolic groups is not required to prove the following classification.

\begin{cor}\label{nosilpair}
    Suppose $\Gamma$ has no SIL. Then the following are equivalent:
    \begin{enumerate}
        \item $\Out(A_{\Gamma})$ is acylindrically hyperbolic;
        \item $\Out(A_\Gamma)$ is virtually free and not virtually cyclic;
        \item $\Out^*(A_{\Gamma})$ is isomorphic to $\SL_2(\Z)$;
        \item $\Gamma-\st(x)$ is connected for each $x\in \Gamma$, and there exist two equivalent vertices $a\sim b$ such that if $v\leq w$ for some distinct vertices $v,w\in \Gamma$, then $\{v,w\}=\{a,b\}$.
    \end{enumerate}
\end{cor}

The following lemma is essential for proving this corollary.

\begin{lem}\cite[Lemma 4.5]{guirardel2018vastness}\label{nilpotent}
    The subgroup $P$ is generated by partial conjugations and transvections $R_{ab}$ for $a\leq b$ but $a\not\sim b$.
\end{lem}

\begin{proof}[Proof of Corollary \ref{nosilpair}]
    ($(2)\Longrightarrow(1)$) If $\Out(A_\Gamma)$ is virtually free, so is hyperbolic. Since non-virtually-cyclic hyperbolic groups are acylindrically hyperbolic, $\Out(A_\Gamma)$ is acylindrically hyperbolic.
    
    ($(3)\Longrightarrow(2)$) Suppose $\Out^*(A_{\Gamma})$ is isomorphic to $\SL_2(\Z)$. Since $\SL_2(\Z)$ is virtually nonabelian free and $\Out^*(A_{\Gamma})$ is a finite index subgroup of $\Out(A_{\Gamma})$, $\Out(A_{\Gamma})$ is also virtually nonabelian free. Therefore, $\Out(A_{\Gamma})$ is virtually free and not virtually cyclic.
    
    ($(1)\Longrightarrow(3)$) Suppose $\Out(A_{\Gamma})$ is acylindrically hyperbolic. Then so is $\Out^*(A_{\Gamma})$. By Lemma \ref{nilpotent}, we know that the chosen generators of $P$ have infinite orders. Since acylindrically hyperbolic groups cannot have finitely generated infinite nilpotent normal subgroups, $P$ must be trivial by Proposition \ref{amenable}. Hence $\Out^*(A_{\Gamma})$ is isomorphic to $\prod_{i=1}^k\SL_{n_i}(\Z)$. By Proposition \ref{product}, at most one $n_i$ can be greater than 1. Thus, we have $\Out^*(A_{\Gamma})\cong \SL_n(\Z)$ for some $n$. It is well-known that $\SL_n(\Z)$ is not acylindrically hyperbolic when $n\geq 3$. Therefore, $\Out^*(A_{\Gamma})$ is acylindrically hyperbolic if and only if $n=2$.
    
    ($(1)\Longrightarrow(4)$) If $\Out(A_{\Gamma})$ is acylindrically hyperbolic, then $P$ is trivial as before. By Lemma \ref{nilpotent}, this implies that there is no nontrivial partial conjugation and no transvection of the form $R_{vw}$ for $v\leq w$, $v\not\sim w$. Thus, $v\leq w$ always implies that $v\sim w$. Furthermore, we can conclude without loss of generality that $n_1=2$ and $n_i=1$ for $i\neq 1$ from the short exact sequence. Hence $v\sim w$ implies that $\{v,w\}=\{a,b\}$. The connectedness of $\Gamma-\st(x)$ for each $x\in\Gamma$ follows from the nonexistence of nontrivial partial conjugations.
    
    ($(4)\Longrightarrow(3)$) The assumption (4) implies that $P$ is trivial, so $\Out^*(A_{\Gamma})$ is isomorphic to $\prod_{i=1}^k\SL_{n_i}(\Z)$. Since there are only two vertices that are equivalent, we can deduce without loss of generality that $n_1=2$ and $n_i=1$ for $i\neq 1$. This implies that $\Out^*(A_{\Gamma})$ is isomorphic to $\SL_2(\Z)$.
\end{proof}

Figure \ref{ExampleofSL2(Z)} illustrates an example satisfying the conditions of Corollary \ref{nosilpair}. It is clear that the complement of the star of each vertex is connected, and that the two tips of the graph induce transvections that generate $\SL_2(\Z)$.

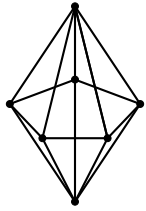
\begin{figure}[ht]
\centering
\begin{tikzpicture}[scale = 1.3]
\draw[black, thick] (-1/3,0.65) -- (0,2);
\draw[black, thick] (1/3,0.65) -- (0,2);
\draw[black, thick] (1/3,0.65) -- (-1/3,0.65);
\draw[black, thick] (1/3,0.65) -- (2/3,1);
\draw[black, thick] (2/3,1) -- (0,2);
\draw[black, thick] (-1/3,0.65) -- (-2/3,1);
\draw[black, thick] (-2/3,1) -- (0,2);
\draw[black, thick] (1/3,0.65) -- (0,2);
\draw[black, thick] (0,5/4) -- (0,2);
\draw[black, thick] (0,5/4) -- (2/3,1);
\draw[black, thick] (0,5/4) -- (-2/3,1);
\draw[black, thick] (0,5/4) -- (0,0);
\draw[black, thick] (2/3,1) -- (0,0);
\draw[black, thick] (-2/3,1) -- (0,0);
\draw[black, thick] (1/3,0.65) -- (0,0);
\draw[black, thick] (-1/3,0.65) -- (0,0);
\filldraw[black] (0,2) circle (1pt);
\filldraw[black] (-1/3,0.65) circle (1pt);
\filldraw[black] (1/3,0.65) circle (1pt);
\filldraw[black] (-2/3,1) circle (1pt);
\filldraw[black] (2/3,1) circle (1pt);
\filldraw[black] (0,5/4) circle (1pt);
\filldraw[black] (0,0) circle (1pt);
\end{tikzpicture}
\caption{An example of $\Gamma$ with $\Out^*(A_{\Gamma})\cong \SL_2(\Z)$.}
\label{ExampleofSL2(Z)}
\end{figure}

\subsection{$\Gamma$ with SIL}

In this subsection, we assume without mentioning that $\Gamma$ has a SIL-pair.

\begin{definition}
    We say that a set of vertices $S=\{w_1,\ldots,w_n\}$ is a maximal SIL-pair system if the following conditions hold:
    \begin{itemize}
        \item $(w_j,w_k)$ is a SIL-pair for every $j\neq k$.
        \item $\cap_{i=1}^{n} \lk(w_i)=\lk(w_j)\cap \lk(w_k)$ for every $j\neq k$.
        \item If $(v,w)$ is a SIL-pair such that $|\lk(v)\cap \lk(w)| < |\cap_{i=1}^{n} \lk(w_i)|$, then there is a path from $v$ to $w$ in $\Gamma - (\lk(v)\cap \lk(w))$.
        \item There is a unique component of $\Gamma - (\cap_{i=1}^{n} \lk(w_i))$ containing $w_i$ and this component does not contain $w_j$ for $j\neq i$.
        \item If $(v,w_i)$ is a SIL-pair for some $i$, then $v$ is contained in a component of $\Gamma - (\cap_{j=1}^{n} \lk(w_j))$ containing some $w_k$.
    \end{itemize}
\end{definition}

Notice that if a maximal SIL-pair system $S=\{w_1,\ldots,w_n\}$ exists, the last condition in the definition of a maximal SIL-pair system implies that $n\geq 2$.

We will prove the existence of a maximal SIL-pair system using the following lemma.

\begin{lem}\label{minimality}
    Let $(v,w)$ be a SIL-pair and let $x$ be a vertex in a shared component of $(v,w)$. Then $\lk(v)\cap \lk(x)\subset \lk(w)$ and $\lk(w)\cap \lk(x)\subset \lk(v)$.
\end{lem}

\begin{proof}
    If there is a vertex $a\in (\lk(v)\cap\lk(x))-\lk(w)$, then the path joining $v$, $a$, and $x$ in a row is contained in $\Gamma-\lk(w)$. However, this is impossible because $v$ and $x$ must be contained in different components of $\Gamma-\lk(w)$. Similarly, we have $\lk(w)\cap \lk(x)\subset \lk(v)$.
\end{proof}

We now prove that a maximal SIL-pair system exists under certain conditions.

\begin{lem}\label{existenceofmaxSILpairsystem}
    Suppose there is a SIL-pair $(v,w)$ in $\Gamma$ such that there is no path between $v$ and $w$ in $\Gamma - (\lk(v)\cap \lk(w))$. Then there is a maximal SIL-pair system.
\end{lem}

\begin{proof}
    We proceed by induction. We first choose a SIL-pair $(w_1,w_2)$ such that $|\lk(w_1)\cap \lk(w_2)|$ is minimal among all SIL-pairs $(v,w)$ for which there is no path between $v$ and $w$ in $\Gamma - (\lk(v)\cap \lk(w))$. Suppose we have constructed a set of vertices $W_k=\{w_1,\ldots,w_k\}$ (for some $k\geq 2$) such that
    \begin{itemize}
        \item $(w_j,w_l)$ is a SIL-pair for every $j\neq l$,
        \item $\cap_{i=1}^{k} \lk(w_i)=\lk(w_j)\cap \lk(w_l)$ for every $j\neq l$,
        \item $|\lk(w_j)\cap \lk(w_l)|$ is minimal among SIL-pairs $(v,w)$ such that there is no path in $\Gamma-(\lk(v)\cap\lk(w))$ between $v$ and $w$, and
        \item there is a unique component of $\Gamma - (\cap_{i=1}^{k} \lk(w_i))$ containing $w_i$, and this component does not contain $w_j$ for $j\neq i$.
    \end{itemize}
    If there is no vertex $w_{k+1}$ in a shared component of $(w_1,w_2)$ that does not contain any $w_i \in W_k$, such that $(w_1,w_{k+1})$ is a SIL-pair, and that there is no path in $\Gamma-(\lk(w_1)\cap\lk(w_{k+1}))$ between $w_1$ and $w_{k+1}$, then $W_k$ is a maximal SIL-pair system. Otherwise, we add such a vertex $w_{k+1}$. By Lemma \ref{minimality} and the minimality condition, one can obtain $\cap_{i=1}^{k+1} \lk(w_i)=\lk(w_j)\cap \lk(w_l)$ for every $j\neq l$. In addition, $(w_i,w_{k+1})$ is a SIL-pair for any $i$, and each $w_i$ lies in a distinct component of $\Gamma-\cap_{i=1}^{k+1}\lk(w_i)$.
    
    Since $\Gamma$ is a finite simplicial graph, this process must terminate in a finite number of steps until the last condition in the definition of a maximal SIL-pair system is satisfied. Therefore, we obtain a maximal SIL-pair system.
\end{proof}

\begin{remark}
    For an illuminating example where a graph has SIL-pairs but no maximal SIL-pair system exists, one can consider the graphs generated by Wiedmer's construction in \cite{wiedmer2024right}, which will be discussed further at the end of this section. In those graphs, every SIL-pair $(v, w)$ has a path between $v$ and $w$ in $\Gamma - (\lk(v)\cap \lk(w))$, so the process of constructing a maximal SIL-pair system cannot even begin.
\end{remark}

If $\Gamma$ has a maximal SIL-pair system $\{w_1,\ldots,w_n\}$, the vertices of $\Gamma$ can be decomposed as follows:
\[\Gamma=(\bigcap_{i=1}^{n}\lk(w_i))\cup (\bigcupdot_{i=1}^{n} C_i)\cup (\bigcupdot_{j=1}^{m}D_j)\]
where $C_i$ is the component of $\Gamma-(\cap_{i=1}^{n}\lk(w_i))$ containing $w_i$, and $D_j$ is a component of $\Gamma-(\cap_{i=1}^{n}\lk(w_i))$ not containing any $w_i$. Each $C_i$ and $D_j$ is called a \textit{simultaneously shared component} and an \textit{additional component}, respectively. Note that the number of additional components can be zero. For instance, if there is a maximal SIL-pair system on a disconnected graph, then no additional component exists. The graph on the left in Figure \ref{examplesofdecompositions} has three simultaneous shared components and no additional component. The other graph has two simultaneous shared components and one additional component.
\begin{figure}[ht]
\centering
\begin{tikzpicture}[scale=0.8]
\draw[black, thick] (-5,0.5) -- (-6.5,0.5);
\draw[black, thick] (-5,0.5) -- (-3.5,0.5);
\draw[black, thick] (-5,1.5) -- (-6.5,0.5);
\draw[black, thick] (-5,1.5) -- (-3.5,0.5);
\draw[black, thick] (-5,-0.5) -- (-6.5,0.5);
\draw[black, thick] (-5,-0.5) -- (-3.5,0.5);
\draw[blue] (-5,1.5) ellipse (0.3cm and 0.3cm) node[anchor=north east] {$C_1$};
\draw[blue] (-5,0.5) ellipse (0.3cm and 0.3cm) node[anchor=north east] {$C_2$};
\draw[blue] (-5,-0.5) ellipse (0.3cm and 0.3cm) node[anchor=north east] {$C_3$};
\filldraw[black] (-5,1.5) circle (1.5pt) node[anchor=west] {$w_1$};
\filldraw[black] (-5,0.5) circle (1.5pt) node[anchor=west] {$w_2$};
\filldraw[black] (-5,-0.5) circle (1.5pt) node[anchor=west] {$w_3$};
\filldraw[black] (-6.5,0.5) circle (1.5pt);
\filldraw[black] (-3.5,0.5) circle (1.5pt);
\draw[black, thick] (0,1) -- (-1.5,0);
\draw[black, thick] (0,1) -- (1.5,0);
\draw[black, thick] (0,-1) -- (-1.5,0);
\draw[black, thick] (0,-1) -- (1.5,0);
\draw[black, thick] (-1,2) -- (-1.5,0);
\draw[black, thick] (1,2) -- (1.5,0);
\draw[black, thick] (-1,2) -- (1,2);
\draw[red] (0,2) ellipse (1.5cm and 0.3cm) node[anchor=south west] {$D_1$};
\draw[blue] (0,1) ellipse (0.3cm and 0.3cm) node[anchor=north east] {$C_1$};
\draw[blue] (0,-1) ellipse (0.3cm and 0.3cm) node[anchor=north east] {$C_2$};
\filldraw[black] (0,1) circle (1.5pt) node[anchor=west] {$w_1$};
\filldraw[black] (0,-1) circle (1.5pt) node[anchor=west] {$w_2$};
\filldraw[black] (-1.5,0) circle (1.5pt);
\filldraw[black] (1.5,0) circle (1.5pt);
\filldraw[red] (1,2) circle (1.5pt);
\filldraw[red] (-1,2) circle (1.5pt);
\end{tikzpicture}
\caption{Examples of decompositions given by maximal SIL-pair systems.}
\label{examplesofdecompositions}
\end{figure}
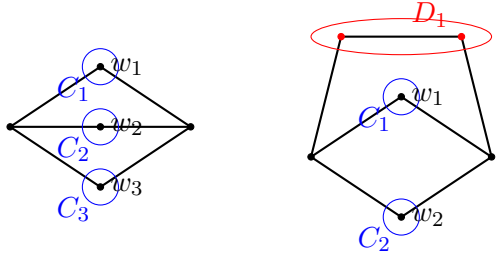

\begin{prop}[Classification of partial conjugations]\label{classofpartialconj}
Let $\Gamma$ be a graph having a maximal SIL-pair system $\{w_1,\ldots,w_n\}$ with the decomposition as above. A partial conjugation $P_v^C$ falls into one of the following categories:
\begin{itemize}
    \item \textbf{Type I: $v \in C_i$ and $\cap_k \lk(w_k) \subset \lk(v)$.}
    \begin{enumerate}
        \item[(1)] $P_v^{C_{i'}}$ for $i'\neq i$.
        \item[(2)] $P_v^{D_j}$ for some $j$.
        \item[(3)] $P_v^C$ where $C\subset C_i$ contains a vertex $w$ such that $\cap_k \lk(w_k)\subset \lk(w)$.
        \item[(4)] $P_v^C$ where $C\subset C_i$ and every vertex $w \in C$ satisfies $\cap_k \lk(w_k)\not\subset \lk(w)$.
    \end{enumerate}
    \item \textbf{Type II: $v \in C_i$ and $\cap_k \lk(w_k) \not\subset \lk(v)$.}
    \begin{enumerate}
        \item[(5)] $P_v^C$ where $C\subset C_i$.
        \item[(6)] $P_v^{D_j}$ for some $j$.
        \item[(7)] $P_v^C$ where $C$ is the unique connected component of $\Gamma-\st(v)$ containing other $C_{i'}$'s for $i'\neq i$.
    \end{enumerate}
    \item \textbf{Type III: $v \in \cap_k \lk(w_k)$.}
    \begin{enumerate}
        \item[(8)] $P_v^C$ where $C$ contains some vertices in $\cap_k \lk(w_k)$.
        \item[(9)] $P_v^C$ where $C$ is a proper subset of $C_i$ for some $i$.
        \item[(10)] $P_v^C$ where $C$ is a proper subset of $D_j$ for some $j$.
    \end{enumerate}
    \item \textbf{Type IV: $v \in D_j$ for some $j$.}
    \begin{enumerate}
        \item[(11)] $P_v^C$ where $C$ is a proper subset of $D_j$.
        \item[(12)] $P_v^{D_{j'}}$ for $j'\neq j$.
        \item[(13)] $P_v^C$ where $C$ is the unique connected component of $\Gamma-\st(v)$ containing all $w_k$'s.
    \end{enumerate}
\end{itemize}
\end{prop}
\begin{proof}
Let us describe the components of $\Gamma-\st(v)$ to classify partial conjugations depending on the decomposition.
\begin{itemize}
    \item \textbf{Type I:} A vertex $v$ in a simultaneously shared component $C_{i}$ may satisfy that $\cap_k \lk(w_k)\subset \lk(v)$. In this case, a component of $\Gamma-\st(v)$ is either $C_{i'}$ for $i'\neq i$, $D_j$, or a proper subset $C\subset C_{i}$. This corresponds to cases (1)-(4).
    \item \textbf{Type II:} If $\cap_k \lk(w_k)\not\subset \lk(v)$ for $v\in C_{i}$, then $\Gamma-\st(v)$ has a unique connected component containing a vertex in $\cap_k \lk(w_k)$. This component contains other $C_{i'}$'s for $i'\neq i$ and some $D_j$'s that are not separated by the vertices in $\cap_k \lk(w_k)$ adjacent to $v$. This means that except for this component, the supports of partial conjugations defined by $v$ are either contained in $C_{i}$ or equal to $D_j$ for some $j$. This corresponds to cases (5)-(7).
    \item \textbf{Type III:} For $v\in \cap_k \lk(w_k)$, components of $\Gamma-\st(v)$ may contain other vertices in $\cap_k \lk(w_k)$. On the other hand, if a component does not contain any vertices in $\cap_k \lk(w_k)$, then it is properly contained in a simultaneously shared component or an additional component. This corresponds to cases (8)-(10).
    \item \textbf{Type IV:} Suppose $v\in D_j$ for some $j$. There is a unique connected component of $\Gamma-\st(v)$ for $v$ in an additional component $D_j$, which contains all $w_k$'s, since $v$ is adjacent to at most one vertex in $\cap_k \lk(w_k)$ by the definition of a maximal SIL-pair system. Except for this component, a component of $\Gamma-\st(v)$ is $D_{j'}$ for $j'\neq j$ or a proper subset of $D_j$. This corresponds to cases (11)-(13).
\end{itemize}
\end{proof}

A partial conjugation of type (1) is called a \textit{dominant partial conjugation}, and type (3) is called a \textit{subordinate partial conjugation}.

The following two criteria demonstrate that $\Out(A_{\Gamma})$ is not acylindrically hyperbolic. By looking at the structure of the maximal SIL-pair system, they detect infinite normal subgroups that are not acylindrically hyperbolic.

\begin{lem}\label{graphwithsingletonintersection}
    Let $\{w_1,\ldots,w_n\}$ be a maximal SIL-pair system in a connected graph $\Gamma$. If $|\cap_i \lk(w_i)|=1$, then $\Out^*(A_{\Gamma})$ has a free-abelian normal subgroup. In particular, $\Out(A_{\Gamma})$ is not acylindrically hyperbolic.
\end{lem}

\begin{proof}
    By Theorem~\ref{kerofprojhom}, it is enough to check that the kernel $P$ of the projection homomorphism is not trivial. Let $\{v\}=\cap_i \lk (w_i)$. If there is a simultaneously shared component $C_i$ whose vertices are adjacent to $v$, then the partial order $w_i\leq v$ generates a leaf transvection. Otherwise, if there is a vertex $v_i$ in each component of $C_i$ that is not adjacent to $v$, then $v_i$ lies in a different $\hat{v}$-component of $\Gamma$. This gives nontrivial $\hat{v}$-component conjugations, which are in $\ker P$.
\end{proof}

\begin{remark}
    Notice that the connectedness of a graph $\Gamma$ in Lemma~\ref{graphwithsingletonintersection} is necessary to apply Theorem~\ref{kerofprojhom}. For example, let $\Gamma$ be the graph shown in Figure~\ref{exampleofdisconnectedgraph}. Then $\{w_1,w_2,w_3\}$ is a maximal SIL-pair system, and $|\cap_i \lk(w_i)|=1$, but $\Gamma$ is not connected. In what follows, we do not assume that $\Gamma$ is connected.
\end{remark}

\begin{figure}[ht]
\centering
\begin{tikzpicture}[scale=0.8]
\draw[black, thick] (0,0.5) -- (-1.5,0.5);
\draw[black, thick] (0,1.5) -- (-1.5,0.5);
\draw[black, thick] (0,-0.5) -- (-1.5,0.5);
\draw[blue] (0,1.5) ellipse (0.3cm and 0.3cm) node[anchor=north east] {$C_1$};
\draw[blue] (0,0.5) ellipse (0.3cm and 0.3cm) node[anchor=north east] {$C_2$};
\draw[blue] (0,-0.5) ellipse (0.3cm and 0.3cm) node[anchor=north east] {$C_3$};
\filldraw[black] (0,1.5) circle (1.5pt) node[anchor=west] {$w_1$};
\filldraw[black] (0,0.5) circle (1.5pt) node[anchor=west] {$w_2$};
\filldraw[black] (0,-0.5) circle (1.5pt) node[anchor=west] {$w_3$};
\draw[red] (1.5,0.5) ellipse (0.3cm and 0.3cm) node[anchor=north west] {$D_1$};
\filldraw[black] (0,1.5) circle (1.5pt);
\filldraw[black] (0,0.5) circle (1.5pt);
\filldraw[black] (0,-0.5) circle (1.5pt);
\filldraw[black] (1.5,0.5) circle (1.5pt);
\filldraw[black] (-1.5,0.5) circle (1.5pt);
\end{tikzpicture}
\caption{An example of a disconnected graph.}
\label{exampleofdisconnectedgraph}
\end{figure}

If there are $m$ additional components, we obtain a normal subgroup isomorphic to the direct product of $m$ copies of a RAAG by collecting all partial conjugations of type (2) as a generating set.

\begin{thm}\label{thm:withaddcpnts}
    Let $\Gamma$ be a graph without isolated vertices. Suppose that $\Gamma$ has a maximal SIL-pair system $\{w_1,\ldots,w_n\}$ and that there are $m$ additional components. Let $P_j$ be a subgroup of $\Out^*(A_\Gamma)$ generated by all partial conjugations of type (2) from Proposition \ref{classofpartialconj}, which are of the form $P_v^{D_j}$ where $v \in C_i$ and $\cap_k \lk(w_k) \subset \lk(v)$. Then all $P_j$'s are isomorphic to the same RAAG and $P_1 \times\cdots\times P_m\lhd \Out^*(A_\Gamma)$.
\end{thm}

\begin{proof}
    Let $\Lambda$ be the induced subgraph whose vertex set consists of the vertices $v$ such that $\Gamma-\st(v)$ has $D_j$ as a connected component. Then a vertex $v$ is contained in $\Lambda$ if and only if $\cap_i \lk(w_i)\subset \lk(v)$. Note that a vertex $v$ with $\cap_i \lk(w_i)\subset \lk(v)$ should be contained in some $C_i$.
    
    The map from $A_\Lambda$ to $P_j$ given by $v\mapsto P_v^{D_j}$ for each $v\in \Lambda$ is an isomorphism between them. Furthermore, elements in $P_j$ and elements in $P_k$ with $j\neq k$ commute. Since $P_j\cap P_k$ is trivial for $j\neq k$, we can see that $\langle P_1,\ldots, P_m\rangle$$=P_1 \times\cdots\times P_m$, so it remains to show that $P_1 \times\cdots\times P_m\lhd \Out^*(A_\Gamma)$.
    
    We first check that the conjugation of a generator of $P_j$ conjugated by a partial conjugation is contained in $P_j$. The link of a vertex in some $D_j$ does not contain $\cap_i \lk(w_i)$ since we choose $\{w_1,\ldots,w_n\}$ as a maximal SIL-pair system. Moreover, if we choose two vertices from $C_i$ and $D_j$ respectively, they cannot be a SIL-pair. By Lemma \ref{lem:commutingrelation} and Proposition \ref{classofpartialconj}, a partial conjugation which does not commute with $P_v^{D_j}$ for $v\in C_i$ is of the form $P_{a}^{C}$ for $\cap_i \lk(w_i)\subset\lk(a)$ and $v\in C$, or $P_{a}^{D_j}\in P_j$ whenever it is defined. For the former case, we have
    \[ P_{a}^{C} P_{v}^{D_j} (P_{a}^{C})^{-1}=(P_{a}^{D_j})^{-1} P_{v}^{D_j} P_{a}^{D_j}\in P_j,\]
    and
    \[ (P_{a}^{C})^{-1} P_{v}^{D_j} (P_{a}^{C})=P_{a}^{D_j} P_{v}^{D_j} (P_{a}^{D_j})^{-1}\in P_j.\]
    
    It remains to consider the conjugations conjugated by transvections. As each vertex in $\cap_i \lk(w_i)$ is adjacent to a vertex in each $C_i$, and each vertex in $D_j$ is adjacent to another vertex in $D_j$, $a\nleq v$ for any vertex $a$ in $D_j\cup (\cap_i \lk(w_i))$ and for any vertex $v\in \Lambda$. Here, we use the assumption that there is no isolated vertex.) Consider a vertex $a\neq v$ in some $C_i$, and suppose that $a\leq v$. Then $[R_{av},P_v^{D_j}]=[L_{av},P_v^{D_j}]=1$ since $a\notin D_j$. For a vertex $a$ such that $v\leq a$, since $\cap_i \lk(w_i)\subset \lk(v)$, $\lk(a)$ also contains $\cap_i \lk(w_i)$, so $a$ belongs to $\Lambda$. Moreover, we have
    \[ R_{va} P_{v}^{D_j} (R_{va})^{-1}=P_{a}^{D_j}P_{v}^{D_j}\in P_j,\]
    \[ (R_{va})^{-1} P_{v}^{D_j} R_{va}=(P_{a}^{D_j})^{-1}P_{v}^{D_j}\in P_j,\]
    \[ L_{va} P_{v}^{D_j} (L_{va})^{-1}=P_{v}^{D_j}P_{a}^{D_j}\in P_j,\]
    and
    \[ (L_{va})^{-1} P_{v}^{D_j} L_{va}=P_{v}^{D_j}(P_{a}^{D_j})^{-1}\in P_j.\]
    Therefore, $P_1 \times\cdots\times P_m$ is a normal subgroup of $\Out^*(A_\Gamma)$.
\end{proof}

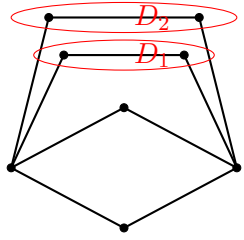
\begin{figure}[ht]
\centering
\begin{tikzpicture}
\draw[black, thick] (0,0.8) -- (-1.5,0);
\draw[black, thick] (0,0.8) -- (1.5,0);
\draw[black, thick] (0,-0.8) -- (-1.5,0);
\draw[black, thick] (0,-0.8) -- (1.5,0);
\draw[black, thick] (-1,2) -- (-1.5,0);
\draw[black, thick] (1,2) -- (1.5,0);
\draw[black, thick] (-1,2) -- (1,2);
\draw[black, thick] (-0.8,1.5) -- (0.8,1.5);
\draw[black, thick] (-0.8,1.5) -- (-1.5,0);
\draw[black, thick] (0.8,1.5) -- (1.5,0);
\draw[red] (0,1.5) ellipse (1.2cm and 0.2cm) node[anchor=west] {$D_1$};
\draw[red] (0,2) ellipse (1.5cm and 0.2cm) node[anchor=west] {$D_2$};
\filldraw[black] (0.8,1.5) circle (1.5pt);
\filldraw[black] (-0.8,1.5) circle (1.5pt);
\filldraw[black] (0,0.8) circle (1.5pt);
\filldraw[black] (0,-0.8) circle (1.5pt);
\filldraw[black] (-1.5,0) circle (1.5pt);
\filldraw[black] (1.5,0) circle (1.5pt);
\filldraw[black] (1,2) circle (1.5pt);
\filldraw[black] (-1,2) circle (1.5pt);
\end{tikzpicture}
\caption{An example of a graph having two additional components.}
\label{twoadditionalcomponents}
\end{figure}

If $\Gamma$ has a maximal SIL-pair system with at least two additional components, by Proposition \ref{normalityAH} and Proposition~\ref{product}, we can conclude that $\Out (A_{\Gamma})$ is not acylindrically hyperbolic. Figure \ref{twoadditionalcomponents} illustrates an example with two additional components. However, if there is at most one additional component, we need a different approach to check acylindrical hyperbolicity. There is a graph with one additional component whose corresponding $\Out^*(A_{\Gamma})$ is acylindrically hyperbolic.

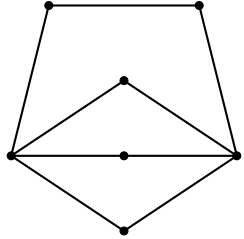
\begin{figure}[ht]
\centering
\begin{tikzpicture}
\draw[black, thick] (0,1) -- (-1.5,0);
\draw[black, thick] (0,1) -- (1.5,0);
\draw[black, thick] (0,0) -- (-1.5,0);
\draw[black, thick] (0,0) -- (1.5,0);
\draw[black, thick] (0,-1) -- (-1.5,0);
\draw[black, thick] (0,-1) -- (1.5,0);
\draw[black, thick] (-1,2) -- (-1.5,0);
\draw[black, thick] (1,2) -- (1.5,0);
\draw[black, thick] (-1,2) -- (1,2);
\filldraw[black] (0,1) circle (1.5pt) node[anchor=north] {$v_1$};
\filldraw[black] (0,0) circle (1.5pt) node[anchor=north] {$v_2$};
\filldraw[black] (0,-1) circle (1.5pt) node[anchor=north] {$v_3$};
\filldraw[black] (-1.5,0) circle (1.5pt) node[anchor=north] {$l_1$};
\filldraw[black] (1.5,0) circle (1.5pt) node[anchor=north] {$l_1$};
\filldraw[black] (1,2) circle (1.5pt);
\filldraw[black] (-1,2) circle (1.5pt);
\end{tikzpicture}
\caption{The graph $\Gamma_1$.}
\end{figure}

\begin{lem}
 $\Out^*(A_{\Gamma_1})$ is acylindrically hyperbolic.
\end{lem}

\begin{proof}
We claim that $\Out^*(A_{\Gamma_1})$ is isomorphic to an acylindrically hyperbolic group $\Aut^*(\mathbb{F}_3)$. Let $v_1$, $v_2$, and $v_3$ be the equivalent vertices in $\Gamma_1$ and $l_1$, $l_2$ be vertices in $\lk (v_1)=\lk (v_2)=\lk (v_3)$. Then $\Out^*(A_{\Gamma_1})$ is generated by transvections defined by the equivalent relations among $v_1$, $v_2$, and $v_3$ since other nontrivial partial conjugations are generated by these transvections. Consider the restriction of the projection homomorphism $P:\Out^* (A_{\Gamma_1}) \to \prod \Out (A_{\lk [v]})$. By Theorem~\ref{kerofprojhom}, $P$ is injective. The images of the projection maps $P_{[v]}$ are trivial other than the cases $v=l_1$ and $v=l_2$. Since $\Out^0(A_{\lk [l_1]})\cong \Out^0(A_{\lk [l_2]})\cong \Out^0(\mathbb{F}_4)$ and the corresponding images of projection maps are isomorphic, we only need to consider the projection map $P_{[l_1]}:\Out^* (A_{\Gamma_1}) \to \Out (A_{\lk [l_1]})$ to describe the image of $P$. Comparing to the restriction of the map $\Aut(\mathbb{F}_n)\to \Out(\mathbb{F}_{n+1})$ to $\Aut^*(\mathbb{F}_n)$, when $n=3$, the images are canonically isomorphic, so we conclude that $\Out^*(A_{\Gamma_1})\cong\Aut^*(\mathbb{F}_3)$.
\end{proof}

\begin{remark}
    In \cite{sale2024virtual}, Sale found an equivalent condition for $\Out(A_{\Gamma})$ to satisfy property (T) when the defining graph has no SIL. Namely, if $\Gamma$ has no SIL, then $\Out(A_\Gamma)$ satisfies property (T) if and only if for any two distinct vertices $u,v$ such that $u\leq v$ there is a vertex $w$ such that $u\leq w\leq v$, and $|X|>1$ for each equivalence class $X$ admitting an $X-$component $C$ for which $(X,C)$ is principal. While $\Gamma_1$ satisfies the latter condition of the previous statement, $\Out(A_{\Gamma_1})$ does not satisfy property (T). This means that if we add the condition of a SIL-pair, Sale's theorem does not hold in general. We don't know if $\Out(A_{\Gamma})$ can be classified as having property (T) with a SIL-pair.
\end{remark}

In \cite{wiedmer2024right}, Wiedmer shows that for any graph $\Lambda$, there is a graph $\Gamma$ without transvections such that $A_{\Lambda}\cong \Out^*(A_{\Gamma})$. As a consequence, since the acylindrical hyperbolicity of the RAAG $A_\Lambda$ is completely understood, one can readily determine the acylindrical hyperbolicity of $\Out^*(A_\Gamma)$. Each vertex of the graphs constructed in Wiedmer's paper generates one or two connected components of the complement of its star. Given $\Gamma$ constructed by Wiedmer's theorem, let $x$ and $y$ be vertices such that both $\Gamma-\st(x)$ and $\Gamma-\st(y)$ have two connected components. Then $(x,y)$ is a SIL-pair. Moreover, $x$ and $y$ lie in the same connected component in $\Gamma-(\lk(x)\cap\lk(y))$, therefore, we cannot apply Lemma \ref{existenceofmaxSILpairsystem}. Thus, the graphs generated by Wiedmer's construction provide illuminating examples where no maximal SIL-pair system exists. We mention the following to describe the graphs without maximal SIL-pair system.

\begin{lem}
    Suppose for each SIL-pair $(v,w)$ in $\Gamma$ there is a path between $v$ and $w$ in $\Gamma - (\lk(v)\cap \lk(w))$. Fix a SIL-pair $(w_1,w_2)$ where $|\lk(w_1)\cap \lk(w_2)|$ is minimal. Then $(x,w_i)$ is not a SIL-pair for any vertex $x$ in components of $\Gamma - (\lk(w_1)\cap \lk(w_2))$ which do not contain $w_1$ and $w_2$.
\end{lem}

\begin{proof}
    Suppose there is a vertex $x$ in a component of $\Gamma - (\lk(w_1)\cap \lk(w_2))$ which do not contain $w_1$ and $w_2$ such that $(x,w_1)$ is a SIL-pair. Then by Lemma \ref{minimality} and the minimality of $|\lk(w_1)\cap \lk(w_2)|$, we obtain that
    \[\lk(w_1)\cap \lk(w_2)=\lk(x)\cap \lk(w_1)=\lk(x)\cap \lk(w_2).\]
    (By these equalities, we deduce that $(x,w_1)$ is a SIL-pair if and only if $(x,w_2)$ is a SIL-pair.) However, it is not possible to get a path between $x$ and $w_1$ in $\Gamma-(\lk(x)\cap \lk(w_1))=\Gamma-(\lk(w_1)\cap \lk(w_2))$, since we assume that $x$ and $w_1$ are in different components of $\Gamma - (\lk(w_1)\cap \lk(w_2))$.
\end{proof}

\begin{lem}
    Suppose for each SIL-pair $(v,w)$ in $\Gamma$ there is a path between $v$ and $w$ in $\Gamma - (\lk(v)\cap \lk(w))$. Fix a SIL-pair $(w_1,w_2)$ where $|\lk(w_1)\cap \lk(w_2)|$ is minimal. If there is a vertex $w_3$ satisfying $\lk(w_1)\cap\lk(w_2)\subset\lk(w_3)$ where the component of $\Gamma - (\lk(w_1)\cap \lk(w_2))$ containing $w_3$ does not contain $w_1$ and $w_2$, then $\Gamma - (\lk(w_1)\cap \lk(w_2))$ has only two connected components.
\end{lem}

\begin{proof}
    If there are at least three connected components, then $(w_1,w_3)$ would be a SIL-pair where $w_1$ and $w_3$ are not joined by a path in $\Gamma - (\lk(w_1)\cap \lk(w_3))$.
\end{proof}

\section{Graphs without additional components} \label{sec:noadditionalcomponents}

\subsection{$\mathrm{PSO}(A_{\Gamma})$ and its acylindrical hyperbolicity}

From now on, we focus on graphs $\Gamma$ without additional components and the pure symmetric outer automorphism groups $\mathrm{PSO}(A_{\Gamma})$. Note that in the existing literature, the notions of a partial conjugation being ``dominating'' or ``subordinate'' are based on reference to a particular SIL-pair. In our context, we classify partial conjugations globally with respect to the maximal SIL-pair system. Specifically, a dominant partial conjugation (type (1) in Proposition~\ref{classofpartialconj}) has its support as a whole simultaneously shared component $C_{i'}$, thus dominating it with respect to the maximal system. A subordinate partial conjugation (type (3) in Proposition~\ref{classofpartialconj}) has its support properly contained inside a simultaneously shared component $C_i$, thus making it subordinate. These definitions are naturally consistent with the classical ones when restricted to the relevant SIL-pairs involving the vertices $w_j$.

\begin{thm} \label{thm:structureofPSO}
    Suppose that a graph $\Gamma$ has a maximal SIL-pair system $\{w_1,\ldots,w_n\}$ without additional components. Let $N_1$ be the subgroup of $\Out(A_{\Gamma})$ generated by all dominant partial conjugations (type (1) in Proposition \ref{classofpartialconj}) except for the partial conjugations of the form $P_{v_1}^{C_2}$ for $v_1\in C_1$ and of the form $P_{v_i}^{C_1}$ for $v_i\in C_i$ with $i\neq 1$. Let $N_2$ be the subgroup of $\Out(A_{\Gamma})$ generated by the other possible partial conjugations described in Proposition~\ref{classofpartialconj} (which are types (4), (5), and (7)-(9), assuming there are no nontrivial subordinate partial conjugations). If there are no nontrivial subordinate partial conjugations (type (3)), then $\mathrm{PSO}(A_{\Gamma})=N_1\times N_2$. In particular, if there are no nontrivial subordinate partial conjugations and no transvections, then $\Out(A_{\Gamma})$ is not acylindrically hyperbolic.
\end{thm}

\begin{proof}
    Since there are no additional components, only partial conjugations of types (1), (3)-(5), and (7)-(9) can exist (types (2), (6), and (10) cannot exist as there is no $D_j$). Furthermore, under the assumption that there are no nontrivial subordinate partial conjugations, type (3) does not exist either. Thus, the only types of partial conjugations present in $\mathrm{PSO}(A_\Gamma)$ are (1), (4), (5), and (7)-(9). Type (1) generators form $N_1$, while the remaining types form $N_2$.
    
    We first show that dominant partial conjugations commute with the generators of $N_2$. Since dominant partial conjugations commute with partial conjugations defined by vertices in $\cap_i \lk (w_i)$ (types (8)-(9)), we only need to check that dominant partial conjugations in $N_1$ commute with partial conjugations of type (4) and (5). Note that the partial conjugations of type (7) equal to the product of the inverses of partial conjugations of type (5).
    
    Let $P_v^C$ be a dominant partial conjugation and $P_w^{C'}$ be a partial conjugation of type (4) or (5). Then $C'$ does not contain $v$ by the classification. Suppose $v$ and $w$ are not adjacent. If not, then $[P_v^C, P_w^{C'}]=1$ by the first relation in Theorem \ref{poutpresentation}. If $v$ and $w$ are contained in different simultaneously shared components, then $[P_v^C, P_w^{C'}]=1$ by the second or third relation in Theorem \ref{poutpresentation} depending on whether $C$ contains $w$ or not. Otherwise $C$ and $C'$ do not contain both $v$ and $w$. Since $P_w^{C'}$ is not a dominant partial conjugation, $C$ is not equal to $C'$. This implies that $C'$ must be another shared component or a subordinate component, so $[P_v^C, P_w^{C'}]=1$ by Lemma~\ref{lem:SharedComponents}.
    
    Some of the generators of the generating set of $\langle N_1, N_2\rangle$ are missing from $\mathrm{PSO}(A_{\Gamma})$. The missing generators are those that were removed in the assumption of the theorem. However, these generators can be identified by the product of the inverses of other partial conjugations defined by the same vertices, so we have $\mathrm{PSO}(A_{\Gamma})=\langle N_1, N_2\rangle$.
    
    To prove that $\langle N_1, N_2\rangle$ represents an internal direct product, consider the map $\mathrm{PSO}(A_{\Gamma})\to N_1$ that sends the generators of $N_1$ to themselves and the generators of $N_2$ to the identity. This map is a well-defined homomorphism because we can treat the presentation of $\mathrm{PSO}(A_\Gamma)$ as having the relation (5) from Theorem \ref{poutpresentation} removed by using it to eliminate the excluded generators of $N_1$. With this reduced presentation, all remaining relations that mix generators of $N_1$ and $N_2$ are solely commutation relations, which are trivially preserved when the $N_2$ generators are evaluated as the identity. The relations involving only $N_1$ generators are preserved identically, and those involving only $N_2$ generators are sent to the trivial relation. The existence of this well-defined homomorphism shows the equality $N_1\cap N_2=\{1\}$. Therefore, $\mathrm{PSO}(A_{\Gamma})=N_1\times N_2$.
\end{proof}

If there are no partial conjugations of type (4) and (5) in the classification theorem, we obtain an interesting group structure of $\mathrm{PSO}(A_{\Gamma})$. We say that a vertex $v\in \Gamma$ \textit{defines only subordinate partial conjugations} if there is no non-trivial partial conjugation in $\Out(A_\Gamma)$ of the form $P_v^C$ of type (4) or (5).

\begin{thm}\label{semidirect}
    Let $\Gamma$ be a graph with a maximal SIL-pair system $\{w_1,\ldots,w_n\}$ and no additional components. Suppose that every vertex in each simultaneously shared component either has the star whose complement is connected, or defines only subordinate partial conjugations, all of whose supports contain a vertex that also defines subordinate partial conjugations. If $n\geq 3$, then
    \[\mathrm{PSO}(A_{\Gamma})=G\times ((H_n\rtimes (H_{n-1}\rtimes(\cdots\rtimes(H_4\rtimes H_3)\cdots)))\rtimes K) \]
    where $G$ is a subgroup generated by partial conjugations defined by vertices in $\cap_i \lk(w_i)$, $H_k$ is a subgroup generated by the union of $\{P_{v_j}^{C_k}:v_j\in C_j, \cap_i \lk(w_i)\subset\lk(v_j), 1\leq j\leq k\}$ and $\{P_{v_k}^{C_j}:v_k\in C_k, \cap_i \lk(w_i)\subset\lk(v_k), 2\leq j\leq k\}$,
    and $K$ is a subgroup generated by all subordinate partial conjugations.
\end{thm}

\begin{proof}
    By the assumptions and the classification of partial conjugations, the standard generating set of $\mathrm{PSO}(A_{\Gamma})$ consists of all subordinate partial conjugations, dominant partial conjugations, and the partial conjugations defined by vertices in $\cap_i \lk(w_i)$. Let $H$ be a subgroup generated by all subordinate and dominant partial conjugations except for dominant partial conjugations of the form $P_{v}^{C_2}$ for $v\in C_1$ or $P_{v}^{C_1}$ for $v\in C_j$ with $j\neq 1$. These partial conjugations are equal to products of the inverse of other dominant partial conjugations and subordinate partial conjugations. Indeed, using Tietze transformations, one can remove the relation (5) in Theorem \ref{poutpresentation}. The union of the generating sets of $G$ and $H$ generates $\Out^*(A_{\Gamma})$ by the choice of generating sets although we remove generators of the form $P_{v}^{C_2}$ for $v\in C_1$ or $P_{v}^{C_1}$ for $v\in C_j$ with $j\neq 1$. Moreover, for each vertex $v$ defining the generators of $H$, we know that $\cap_i \lk(w_i)\subset\lk(v)$ by the definition of simultaneously partial conjugation and subordinate partial conjugation. This implies that the generators of $G$ commute with the generators of $H$.
    
    To check that $G\cap H$ is trivial, consider a map $\mathrm{PSO}(A_{\Gamma})\to H$ sending standard generators of $H$ to themselves and standard generators of $G$ to the identity. Since this map sends all relations in Theorem \ref{poutpresentation} to the identity, it is a well-defined homomorphism. While every nontrivial element in $H$ maps to itself, every nontrivial element in $G$ maps to the trivial element. This implies that $G\cap H$ is trivial. Therefore, we deduce that $\mathrm{PSO}(A_{\Gamma})=G\times H$.
    
    Let $H'$ be a subgroup generated by all dominant partial conjugations except for dominant partial conjugations of the form $P_{v}^{C_2}$ for $v\in C_1$ and of the form $P_{v}^{C_1}$ for $v\in C_j$ with $j\neq 1$. It remains to show that $H=H'\rtimes K$ and $H'=H_n\rtimes (H_{n-1}\rtimes(\cdots\rtimes(H_4\rtimes H_3)\cdots))$. We need to check that $H'$ is a normal subgroup of $H$.
    
    \begin{lem}\label{normality1}
        $H'$ is a normal subgroup of $H$.
    \end{lem}
    
    \begin{proof}[Proof of Lemma \ref{normality1}]
        Let $v\in C_i$ be a vertex defining dominant partial conjugations and let $P_{v}^{C_l}$ be a dominant partial conjugation where $l\neq i$. Choose another subordinate partial conjugation $P_{w}^{D}$. To check the normality, it suffices to check that $P_{w}^{D} P_{v}^{C_l} (P_{w}^{D})^{-1}\in H'$ and $(P_{w}^{D})^{-1} P_{v}^{C_l} P_{w}^{D}\in H'$. There are only four possibilities and by using (4) in Theorem~\ref{poutpresentation}, we get
        \begin{enumerate}
            \item If $w\in C_{l}$, then $P_{w}^{D} P_{v}^{C_l} (P_{w}^{D})^{-1}=(P_{w}^{D})^{-1} P_{v}^{C_l} P_{w}^D=P_{v}^{C_l}$ since $D\subset C_{l}$.
            \item If $w\notin C_{i}$ and $w\notin C_{l}$, then $P_{w}^{D} P_{v}^{C_l} (P_{w}^{D})^{-1}=(P_{w}^{D})^{-1} P_{v}^{C_l} P_{w}^D=P_{v}^{C_l}$ since $D\not\subset C_{i}$ and $D\not\subset C_{l}$.
            \item If $w\in C_i$ and $v\notin D$, then $P_{w}^{D} P_{v}^{C_l} (P_{w}^{D})^{-1}=(P_{w}^{D})^{-1} P_{v}^{C_l} P_{w}^D=P_{v}^{C_l}$.
            \item If $w\in C_i$ and $v\in D$, then $P_{w}^{D} P_{v}^{C_l} (P_{w}^{D})^{-1}=(P_{w}^{C_l})^{-1}P_{v}^{C_l}P_{w}^{C_l}\in H'$, and $(P_{w}^{D})^{-1} P_{v}^{C_l} P_{w}^D=P_{w}^{C_l} P_{v}^{C_l} (P_{w}^{C_l})^{-1}\in H'$.
        \end{enumerate}
        Therefore, $H'\trianglelefteq H$.
    \end{proof}
    The last two relations in the proof of Lemma \ref{normality1} imply that for any element in $H$, all subordinate partial conjugations can be moved from left to right in a word presentation of the element. Thus we have $H' K=H$. To see that $H'\cap K$ is trivial, as before, consider a map $\mathrm{PSO}(A_{\Gamma})\to K$ sending standard generators of $K$ to themselves and other standard generators of $\mathrm{PSO}(A_{\Gamma})$ to the identity. By looking at the relations in Theorem \ref{poutpresentation} again, it is a well-defined homomorphism, so $H'\cap K=\{1\}$.
    
    Let $H_{l}'$ be a subgroup of $\mathrm{PSO}(A_{\Gamma})$ generated by the union of the generating sets of $H_j$ for $3 \leq j \leq l$ that we choose. Since $H_3 '=H_3$, the last things that we need to check are $H'=H_n\rtimes H_{n-1}'$ and $H_{l}'=H_{l}\rtimes H_{l-1}'$ for $4\leq l \leq n-1$.
    
    \begin{lem}\label{normality2}
        $H_n$ is a normal subgroup of $H'$, and $H_l$ is a normal subgroup of $H_l '$.
    \end{lem}
    
    \begin{proof}[Proof of Lemma \ref{normality2}]
    The proof that $H_l\trianglelefteq H_l '$ is the same as the proof that $H_n\trianglelefteq H'$ by changing $n$ for $l$, so we only show that $H_n\trianglelefteq H'$. For each $1\leq l \leq n$, let $v_l\in C_l$ be a vertex defining dominant partial conjugations. For $i,j\neq n$ and $k\neq j,n$, it suffices to check that $P_{v_j}^{C_k} P_{v_i}^{C_n} (P_{v_j}^{C_k})^{-1}\in H_n$, $(P_{v_j}^{C_k})^{-1} P_{v_i}^{C_n} P_{v_j}^{C_k}\in H_n$, $P_{v_j}^{C_k} P_{v_n}^{C_i} (P_{v_j}^{C_k})^{-1}\in H_n$, and $(P_{v_j}^{C_k})^{-1} P_{v_n}^{C_i} P_{v_j}^{C_k}\in H_n$. There are several possibilities:
        \begin{enumerate}
            \item $P_{v_j}^{C_k} P_{v_i}^{C_n} (P_{v_j}^{C_k})^{-1}=(P_{v_j}^{C_k})^{-1} P_{v_i}^{C_n} P_{v_j}^{C_k}=P_{v_i}^{C_n}$ if $k\neq i$
            \item $P_{v_j}^{C_i} P_{v_i}^{C_n} (P_{v_j}^{C_i})^{-1}=(P_{v_j}^{C_n})^{-1}P_{v_i}^{C_n} P_{v_j}^{C_n}\in H_n$ and\\
            $(P_{v_j}^{C_i})^{-1} P_{v_i}^{C_n} P_{v_j}^{C_i}=P_{v_j}^{C_n}P_{v_i}^{C_n} (P_{v_j}^{C_n})^{-1}\in H_n$
            \item $P_{v_j}^{C_k} P_{v_n}^{C_i} (P_{v_j}^{C_k})^{-1}=(P_{v_j}^{C_k})^{-1} P_{v_n}^{C_i} P_{v_j}^{C_k}=P_{v_n}^{C_i}$ if $j,k\neq i$
            \item $P_{v_j}^{C_i} P_{v_n}^{C_i} (P_{v_j}^{C_i})^{-1}= (P_{v_j}^{C_n})^{-1} P_{v_n}^{C_i} P_{v_j}^{C_n}\in H_n$ and\\
            $(P_{v_j}^{C_i})^{-1} P_{v_n}^{C_i} P_{v_j}^{C_i}$
            $= P_{v_j}^{C_n} P_{v_n}^{C_i} (P_{v_j}^{C_n})^{-1}\in H_n$
            \item $P_{v_i}^{C_k} P_{v_n}^{C_i} (P_{v_i}^{C_k})^{-1}=P_{v_n}^{C_i} P_{v_n}^{C_k}
                (P_{v_i}^{C_n})^{-1} (P_{v_n}^{C_k})^{-1} P_{v_i}^{C_n}\in H_n$ and\\
            $(P_{v_i}^{C_k})^{-1} P_{v_n}^{C_i} P_{v_i}^{C_k}= P_{v_n}^{C_i} P_{v_n}^{C_k}
                P_{v_i}^{C_n} (P_{v_n}^{C_k})^{-1} (P_{v_i}^{C_n})^{-1}\in H_n$
        \end{enumerate}
        The last two cases come from these equalities:
        \begin{align*}
                P_{v_i}^{C_k} P_{v_n}^{C_i} (P_{v_i}^{C_k})^{-1}&=P_{v_n}^{C_i} (P_{v_n}^{C_i})^{-1}
                P_{v_i}^{C_k} P_{v_n}^{C_i} (P_{v_i}^{C_k})^{-1}\\
                 &=P_{v_n}^{C_i} P_{v_n}^{C_k}
                P_{v_i}^{C_k} (P_{v_n}^{C_k})^{-1} (P_{v_i}^{C_k})^{-1}\\
                 &=P_{v_n}^{C_i} P_{v_n}^{C_k}
                (P_{v_i}^{C_n})^{-1} (P_{v_n}^{C_k})^{-1} P_{v_i}^{C_n},
        \end{align*}
        \begin{align*}
                (P_{v_i}^{C_k})^{-1} P_{v_n}^{C_i} P_{v_i}^{C_k}&=P_{v_n}^{C_i} (P_{v_n}^{C_i})^{-1}
                (P_{v_i}^{C_k})^{-1} P_{v_n}^{C_i} P_{v_i}^{C_k}\\
                 &=P_{v_n}^{C_i} P_{v_n}^{C_k}
                (P_{v_i}^{C_k})^{-1} (P_{v_n}^{C_k})^{-1} P_{v_i}^{C_k}\\
                 &=P_{v_n}^{C_i} P_{v_n}^{C_k}
                P_{v_i}^{C_n} (P_{v_n}^{C_k})^{-1} (P_{v_i}^{C_n})^{-1}.
        \end{align*}
        Therefore, $H_n\trianglelefteq H'$ and $H_l\trianglelefteq H_l'$.
    \end{proof}
    
    The relations in Lemma \ref{normality2} enable us to move the generators in a word presentation of an element in $H'$ from left to right. Thus we have $H_n H_{n-1}'=H'$ and $H_l H_{l-1}'=H_l'$. To see that $H_n \cap H_{n-1}'$ and $H_l\cap H_{l-1}'$ are trivial, one can consider well-defined maps $\mathrm{PSO}(A_{\Gamma})\to \langle K, H_{n-1}'\rangle$ and $\mathrm{PSO}(A_{\Gamma})\to \langle K, H_{l-1}'\rangle$ defined by sending the standard generators of the codomains to themselves, and other generators to the identity. The existence of these maps gives the triviality of both $H_n \cap H_{n-1}'$ and $H_l\cap H_{l-1}'$.
\end{proof}

Consider a graph with a maximal SIL-pair system containing three simultaneous shared components. We state a theorem that $\mathrm{PSO}(A_{\Gamma})$ is acylindrically hyperbolic in certain cases. To see this, we first show that $\mathrm{PSO}(A_{\Gamma})$ can be viewed as an HNN-extension group.

\begin{lem}[\cite{osin2015acylindrical}]\label{HNN}
 Let $G$ be a finitely presented group admitting a subjective homomorphism $G\to \mathbb{Z}$. Then $G$ is an HNN-extension of a finitely generated group with finitely generated associated subgroups.
\end{lem}

  \begin{remark}\label{rem:associatedsubgroups}
  According to the proof of Lemma~\ref{HNN}, one can describe the structure of the splitting of $G$. Choose an element of $G$ sent to $\pm 1$ in $\Z$. Then $G$ can be written as a finite presentation
  \[G=\langle t,a_1,\ldots,a_n \, | \, R_1,\ldots, R_m\rangle.\]
  Note that after taking conjugation by a suitable power of $t$, every $R_i$ can be written as a product of elements of the form $t^{\beta}a_{\alpha}^{\pm 1} t^{-\beta}$, where $\beta$ is a nonnegative integer. Denote $b_{\alpha \beta}=t^{\beta}a_{\alpha}^{\pm 1} t^{-\beta}$, and let $B$ be a set of all $b_{\alpha \beta}$ appeared in $\{R_1,\ldots,R_m\}$ as a letter. Then we obtain a new presentation of $G$
  \[G=\langle t, B \, | \, S_1,\ldots, S_m, \mathcal{T}\rangle,\]
  where $S_i$ is obtained from $R_i$ after the rewriting using $b_{\alpha \beta}$'s and $\mathcal{T}$ is the set of all relations of the form
  \[b_{\alpha,\beta+1}=t b_{\alpha,\beta} t^{-1}.\]
  Therefore, the associated subgroups are $\{b_{\alpha\beta} \, | \,1\leq \alpha \leq n,\, 0\leq \beta\leq N-1\}$ and $\{b_{\alpha\beta} \, | \,1\leq \alpha \leq n,\, 1\leq \beta\leq N\}$, where $N$ is the maximal number of occurrences of $t$ in the relators $\{R_1,\ldots,R_m\}$. We use this structure to prove our theorems subsequently.
  \end{remark}
  
We explain why we only focus on the subgroup $\mathrm{PSO}(A_{\Gamma})$ instead of the whole group $\Out(A_{\Gamma})$. Even though $\mathrm{PSO}(A_{\Gamma})$ can be written as an HNN-extension with stable letter $t$, $\Out(A_{\Gamma})$ may not be an HNN-extension with the same stable letter. For example, suppose $\Gamma$ is a graph with a maximal SIL-pair system containing at least three simultaneously shared components. Let $t=P_{w_1}^{C_3}$ be the stable letter, and let $\iota_{w_1}$ be the inversion defined by $w_1\mapsto w_1^{-1}$ and fixing other generators. Then $\iota_{w_1} t\iota_{w_1}^{-1}=t^{-1}$, and since $\iota_{w_1}$ has order 2, the map $\Out(A_{\Gamma})\to \mathbb{Z}$ sending $t$ to $1$ is not well-defined.
  
We may try to apply Minasyan-Osin's theorem to decide acylindrical hyperbolicity of an HNN-extension group.
  
\begin{lem}[\cite{minasyan2015acylindrical}]\label{minasyanosin}
    Let $\widetilde{G}$ be an HNN-extension of a group $G$ with associated subgroups $A$ and $B$. Suppose that $A\neq G \neq B$ and there is an element $g\in \widetilde{G}$ such that $|gAg^{-1} \cap A|<\infty$. Then $\widetilde{G}$ is acylindrically hyperbolic.
\end{lem}

If $|gAg^{-1} \cap A|<\infty$ holds for some $g\in \widetilde{G}$, then we say that $A$ is weakly malnorlmal in $\widetilde{G}$. We also say that the HNN-extension is non-ascending if $A\neq G\neq B$. In our cases, the following lemma is useful to prove weak malnormality.

\begin{lem}[Britton's lemma, \cite{miller2002combinatorial}]
    Let $\widetilde{G}$ be an HNN-extension of $G$ with associated subgroups $A$ and $B$. Let $t$ be the stable letter.  If a word $w=g_0 t^{\epsilon_1} g_1 t^{\epsilon_2} \cdots g_{n-1} t^{\epsilon_n} g_n$ with $g_i\in G$ and with $\epsilon_i=\pm 1$ is equal to the identity in $\widetilde{G}$, then either $n=0$ and $g_0=1$ in $G$, or $w$ contains a subword of the form $tat^{-1}$ where $a\in A$, or of the form $t^{-1}bt$ where $b\in B$.
\end{lem}

We now state the theorem proving that $\mathrm{PSO}(A_{\Gamma})$ is acylindrically hyperbolic in certain cases.

\begin{thm}\label{beingAH}
    Let $\Gamma$ be a graph with a maximal SIL-pair system $\{w_1,w_2,w_3\}$ and no additional components. Suppose the following conditions hold:
    \begin{enumerate}
        \item Each vertex in each simultaneously shared component either has the star whose complement is connected, or defines only subordinate partial conjugations, all of whose supports contain a vertex that also defines subordinate partial conjugations.
        \item The map $\varphi : K\to\Aut(H_3)$ given by Theorem \ref{semidirect} is injective.
    \end{enumerate}
    Then $\mathrm{PSO}(A_{\Gamma})$ is acylindrically hyperbolic if and only if there is no nontrivial partial conjugation defined by vertices in $\cap_i \lk(w_i)$.
\end{thm}

\begin{proof}
    If there is no nontrivial partial conjugation defined by vertices in $\cap_i \lk(w_i)$, then by Proposition \ref{product} and Theorem \ref{semidirect}, $\mathrm{PSO}(A_{\Gamma})$ is not acylindrically hyperbolic, since $G$ is non-trivial. This proves the ``if" part.
    
    To prove the ``only if" part, let $v_{1,1}=w_1,\ldots,v_{1,{n_1}}\in C_1$, $v_{2,1}=w_2,\ldots,v_{2,{n_2}}\in C_2$, and $v_{3,1}=w_3,\ldots,v_{3,{n_3}}\in C_3$ be vertices defining subordinate partial conjugations. Without loss of generality, we assume that $n_1 \leq n_2 \leq n_3$. Note that $n_1\geq 2$ holds by our assumption in the theorem.
    
    By Theorem~\ref{semidirect}, $\mathrm{PSO}(A_{\Gamma})$ is generated by all subordinate partial conjugations, dominant partial conjugations of the form $P_{v_1}^{C_3}$ for $v_1\in C_1$, $P_{v_2}^{C_3}$ for $v_2\in C_2$, and $P_{v_3}^{C_2}$ for $v_3\in C_3$ whenever they are defined. Indeed, the missing generators of the form $P_{v_1}^{C_2}$ for $v_1\in C_1$, $P_{v_2}^{C_1}$ for $v_2\in C_2$, and $P_{v_3}^{C_1}$ for $v_3\in C_3$ can be represented by the product of inverses of other generators, so we can remove the last relation in Theorem~\ref{poutpresentation} using Tietze transformation. Once we forget these missing generators, defining a map $\mathrm{PSO}(A_{\Gamma})\to \Z$ given by sending $P_{w_1}^{C_3}$ to 1 and other generators to $0$, this leads us to apply Lemma~\ref{HNN} with the stable letter $t=P_{w_1}^{C_3}$. By the arguments in Remark~\ref{rem:associatedsubgroups}, the associated subgroups $A=B$ are generated by all subordinate partial conjugations and dominant partial conjugations of the form $P_{v_1}^{C_3}$ for $v_1\in C_1$ except $P_{w_1}^{C_3}$. The existence of a map from $\mathrm{PSO}(A_{\Gamma})$ to the infinite cyclic group generated by a generator, which is not in the generating set of $A$, ensures that the given HNN-extension is non-ascending.
    
    Let
    \[P_1^i=P_{v_{2,i}}^{C_3}P_{v_{1,i}}^{C_3}P_{v_{3,i}}^{C_2},\quad P_2^i=P_{v_{2,i}}^{C_3}P_{v_{1,n_1}}^{C_3}P_{v_{3,i}}^{C_2},\quad P_3^i=P_{v_{2,{n_2}}}^{C_3}P_{v_{1,n_1}}^{C_3}P_{v_{3,i}}^{C_2},\]
    and let
    \[g=(\prod _{i=1}^{n_1} P_1^i) (\prod _{i=n_1 +1}^{n_2}P_2^i) (\prod_{i=n_2+1}^{n_3} P_3^i).\]
    Now we show that $gAg^{-1}\cap A$ is trivial.
    
    Choose $a\in A$ so that $gag^{-1}\in A$. Then $gag^{-1}a'=1$ for some $a'\in A$. Note that there is only one $t$ and one $t^{-1}$ in the word presentation of $gag^{-1}a'$. Thus Britton's lemma implies that $g_1 a g_{1}^{-1}\in A$ where
    \[g_1=P_{w_3}^{C_2}(\prod _{i=2}^{n_1} P_1^i)(\prod _{i=n_1 +1}^{n_2}P_2^i) (\prod_{i=n_2+1}^{n_3} P_3^i).\]
    Then $g_1 a g_1^{-1} a''=1$ for some $a''\in A$. By the same idea of the interpretation of $\mathrm{PSO}(A_{\Gamma})$ as the HNN-extension with the stable letter $t=P_{w_1}^{C_3}$, we can choose another stable letter $t_1=P_{w_3}^{C_2}$ and associate subgroups $A_1=B_1$ generated by all subordinate partial conjugations and dominant partial conjugations of the form $P_{v_3}^{C_2}$ for $v_3\in C_3$ except $P_{w_3}^{C_2}$. Applying Britton's lemma again, we have $g_2 a g_2^{-1}\in A_1$ where
    \[g_2=(\prod _{i=2}^{n_1} P_1^i)(\prod _{i=n_1 +1}^{n_2}P_2^i) (\prod_{i=n_2+1}^{n_3} P_3^i)\]
    since there is no letter $t_1$ in $a$ and $a''$, so there is only one $t_1$ and one $t_1^{-1}$ in the word presentation of $g_1 a g_1^{-1}a''$. Now we know that $g_2 a g_2^{-1} a_1=1$ for some $a_1\in A_1$. Letting $t_2=P_{v_{2,2}}^{C_3}$, $\mathrm{PSO}(A_{\Gamma})$ is the HNN-extension with the stable letter $t_2$ and associated subgroups $A_2=B_2$ generated by all subordinate partial conjugations and dominant partial conjugations of the form $P_{v_2}^{C_3}$ for $v_2\in C_2$ except $P_{v_{2,2}}^{C_3}$. By the same idea, we have that $g_3 a g_3^{-1}\in A_2$ where
    \[g_3=P_{v_{1,2}}^{C_3}P_{v_{3,2}}^{C_2}(\prod _{i=3}^{n_1} P_1^i)(\prod _{i=n_1 +1}^{n_2}P_2^i) (\prod_{i=n_2+1}^{n_3} P_3^i).\] Then $g_3 a g_3^{-1}a_2=1$ for some $a_2\in A_2$.
    
    Letting $t_3=P_{v_{1,2}}^{C_3}$, $\mathrm{PSO}(A_{\Gamma})$ is the HNN-extension with the stable letter $t_3$ and associated subgroups $A_3=B_3$ generated by all subordinate partial conjugations and dominant partial conjugations of the form $P_{v_1}^{C_3}$ for $v_1\in C_1$ except $P_{v_{1,2}}^{C_3}$. To apply Britton's lemma, one must be careful that $A$ contains $P_{v_{1,2}}^{C_3}$ so that there is more than one $t_3$ or $t_3^{-1}$ in the word presentation of $g_3 a g_3^{-1}$. Removing subwords of the form $t_3 a_3 t_3^{-1}$ where $a_3\in A_3$ and of the form $t_3^{-1} b_3 t_3$ where $b_3\in B_3$, we can see that $a$ can be written as
    \[a=x_0 t_3^{\epsilon_1} x_1 t_3^{\epsilon_2}\cdots x_{n-1}t_3^{\epsilon_n}x_n\]
    with $x_i$ a product of subordinate partial conjugations and dominant partial conjugations of the form $P_{v_1}^{C_3}$ for $v_1\in C_1$ except $P_{w_1}^{C_3}$ and $P_{v_{1,2}}^{C_3}$, $\epsilon_i=\pm 1$ and without subwords of the form $t_3 a_3 t_3^{-1}$ where $a_3\in A_3$ and of the form $t_3^{-1} b_3 t_3$ where $b_3\in B_3$. Suppose at least one $t_3$ or one $t_3^{-1}$ exists in the word expression of $a$. We know that
    \[g_4=P_{v_{3,2}}^{C_2}(\prod _{i=3}^{n_1} P_1^i)(\prod _{i=n_1 +1}^{n_2}P_2^i) (\prod_{i=n_2+1}^{n_3} P_3^i)\notin A_3\]
    so that there are no subwords of the form $t_3 a_3 t_3^{-1}$ where $a_3\in A_3$ and of the form $t_3^{-1} b_3 t_3$ where $b_3\in B_3$ in the presentation $g_3 a g_3^{-1} a_2$. It contradicts Britton's lemma, so there is no letter $t_3$ or $t_3^{-1}$ in the word presentation of $a$. This also implies that $g_4 a g_4^{-1}\in A_3$. Applying this procedure several times, we conclude that $a\in A'$ where $A'$ is generated by all subordinate partial conjugations and dominant partial conjugations of the form $P_{v_3}^{C_2}$ for $v_3\in C_3$ except $P_{v_{3,n_3}}^{C_2}$. By uniqueness of the expression of an element in a semi-direct product group, $a\in A\cap A'=K$ since both $A$ and $A'$ contain $K$ and do not share the same dominant partial conjugations.
    
    We also have that $P_{v_{3,n_3}}^{C_2} a (P_{v_{3,n_3}}^{C_2})^{-1}\in P_{v_{3,n_3}}^{C_2} K (P_{v_{3,n_3}}^{C_2})^{-1}\cap A''$ where $A''$ is generated by all subordinate partial conjugations and dominant partial conjugations of the form $P_{v_1}^{C_3}$ for $v_1\in C_1$ except $P_{v_{1,n_1}}^{C_3}$. By uniqueness of the expression of elements in the semi-direct product group, we have $P_{v_{3,n_3}}^{C_2} K (P_{v_{3,n_3}}^{C_2})^{-1}\cap A''\subset K$, and hence $P_{v_{3,n_3}}^{C_2} a (P_{v_{3,n_3}}^{C_2})^{-1}\in K$. Since
    \[P_{v_{3,n_3}}^{C_2} a (P_{v_{3,n_3}}^{C_2})^{-1}=P_{v_{3,n_3}}^{C_2} \varphi(a)((P_{v_{3,n_3}}^{C_2})^{-1})a,\]
    we have
    \[\varphi(a)(P_{v_{3,n_3}}^{C_2})=aP_{v_{3,n_3}}^{C_2}a^{-1}=P_{v_{3,n_3}}^{C_2},\]
    so $P_{v_{3,n_3}}^{C_2} a (P_{v_{3,n_3}}^{C_2})^{-1}=a$. By the same reason, we get $P_{v_{1,n_1}}^{C_3} a (P_{v_{1,n_1}}^{C_3})^{-1}\in K$ and so $P_{v_{1,n_1}}^{C_3} a (P_{v_{1,n_1}}^{C_3})^{-1}=a$. We can conclude that $\varphi(a)$ fixes every generator in $H_3$ so that $\varphi(a)$ is the trivial automorphism. Therefore, the identity is the only element in $gAg^{-1}\cap A$. By Lemma \ref{minasyanosin}, $\mathrm{PSO}(A_{\Gamma})$ is acylindrically hyperbolic.
\end{proof}

\begin{remark}
    If there are more than three simultaneously shared components, then the associated subgroups may have more generators compared to the previous case. For example, for $n\geq 4$, if we fix $t=P_{w_1}^{C_n}$ as the stable letter of an HNN-extension, then the generators of the form $P_{v_i}^{C_j}$ with $i,j\neq 1,n$ commute with $t$. The ideas applying Britton's lemma do not hold or become more complicated.
\end{remark}

  It is tempting to speculate that the injectivity of the map $\varphi : K\to \Aut(H_3)$ from Theorem \ref{beingAH} automatically holds. However, it is not easy to control all subordinate partial conjugations only by consideration of relations. Instead, we give a corollary of the theorem when it is evident that $\ker \varphi$ is trivial.
  
\begin{cor}\label{beingAHcor}
Let $\Gamma$ be a graph with a maximal SIL-pair system $\{w_1,w_2,w_3\}$ and no additional components. Suppose the following conditions hold:
    \begin{enumerate}
        \item Each vertex in each simultaneously shared component either has the star whose complement is connected, or defines only subordinate partial conjugations, all of whose supports contain a vertex that also defines subordinate partial conjugations.
        \item For each subordinate partial conjugation $P_{v_i}^C$ for $v_i\in C_i$, $C$ contains all vertices in $C_i$ except $v_i$ that define subordinate partial conjugations.
    \end{enumerate}
    Then $\mathrm{PSO}(A_{\Gamma})$ is acylindrically hyperbolic if and only if there is no nontrivial partial conjugation defined by vertices in $\cap_i \lk(w_i)$.
\end{cor}

\begin{proof}
    Consider the map $\varphi : K\to \Aut(H_3)$ given by Theorem \ref{semidirect}. Let $P_{v_i}^C$ be the subordinate partial conjugation with $v_i\in C_i$. By the assumption, $v_i$ generates exactly one subordinate partial conjugation. Then
    \[\varphi(P_{v_i}^{C})(P_{v_{i}'}^{C_j})
    =\begin{cases}(P_{v_{i}}^{C_j})^{-1}P_{v_{i}'}^{C_j}P_{v_{i}}^{C_j}, &\text{if $v_{i}'\in C_i$},\\ P_{v_{i}'}^{C_j}, &\text{otherwise.}\end{cases}
    \]
    To prove that $\ker \varphi$ is trivial, since two subordinate partial conjugations whose supports are included in different simultaneous shared component commute, without loss of generality, it is enough to show that for each product $P$ of subordinate partial conjugations whose supports are included in $C_1$, $\varphi(P)=id_{A_{\Gamma}}$ implies that $P=id_{A_{\Gamma}}$. We first collect all vertices $v_1,\ldots,v_n$ defining subordinate partial conjugations $P_{v_1}^{C^1},\ldots,P_{v_n}^{C^n}$ whose supports are included in $C_1$. Let $P=(P_{v_{i_1}}^{C^{i_1}})^{\epsilon_1}\cdots(P_{v_{i_k}}^{C^{i_k}})^{\epsilon_k}$ where $\epsilon=\pm 1$. Then
    \[\varphi(P)(P_{v_i}^{C_3})=(P_{v_{i_1}}^{C_3})^{-\epsilon_1}\cdots(P_{v_{i_k}}^{C_3})^{-\epsilon_k}P_{v_i}^{C_3}(P_{v_{i_k}}^{C_3})^{\epsilon_k}\cdots(P_{v_{i_1}}^{C_3})^{\epsilon_1}.\]
    Note that $\langle P_{v_1}^{C_3}\ldots,P_{v_n}^{C_3}\rangle\cong \mathbb{F}_n$ as $\langle v_1,\ldots,v_n\rangle\cong \mathbb{F}_n$. Thus for each $i$ the centralizer of $P_{v_i}^{C_3}$ in $K$ is the cyclic group generated by $P_{v_i}^{C_3}$. This means that $P$ should be the identity if $\varphi(P)$ fixes all $P_{v_i}^{C_3}$. Therefore, $\ker \varphi$ is trivial and the corollary follows from Theorem~\ref{beingAH}.
\end{proof}

\begin{remark}
    One may apply Genevois' theorem in \cite{genevois2019negative} to prove acylindrical hyperbolicity of semi-direct product groups. In our cases, however, it is challenging to prove acylindrical hyperbolicity of the subgroup generated by dominant partial conjugations outlined in Theorem \ref{semidirect}. In the case of Corollary \ref{beingAHcor}, if it can be shown that $H_3$ is a free group, then it can be concluded that $\mathrm{PSO}(A_{\Gamma})$ is acylindrically hyperbolic by Genevois' theorem. While it appears that $H_3$ is a free group, but the authors have been unable to identify a general approach for proving this.
    
    In the case where there are precisely two simultaneous shared components, we know that $\Out(A_{\Gamma})$ is not acylindrically hyperbolic in the absence of transvections, despite the absence of nontrivial partial conjugations defined by vertices in $\cap_i \lk(w_i)$. In this case, the dominant partial conjugation is equal to the product of inverses of subordinate partial conjugations, so $\mathrm{PSO}(A_{\Gamma})$ can be written as a product of two infinite subgroups.
\end{remark}

\subsection{Examples} \label{subsec:examples}

Let $\Gamma_2$ be the graph on the left in Figure \ref{gamma23}. In this graph, we can choose a maximal SIL-pair system $\{w_1, w_2, w_3\}$ consisting of one vertex from each of the three vertical branches. The intersection of their links is exactly the set of two horizontal extremal vertices $l_1$ and $l_2.$ Note that there are three simultaneously shared components and no additional components. $\Out^*(A_{\Gamma_2})$ has partial conjugations defined by (two) vertices in $\cap \lk(w_i)$, and by Theorem \ref{semidirect}, $\Out^*(A_{\Gamma_2})$ is a direct product of two infinite subgroups. Therefore, $\Out^*(A_{\Gamma_2})$ is not acylindrically hyperbolic.
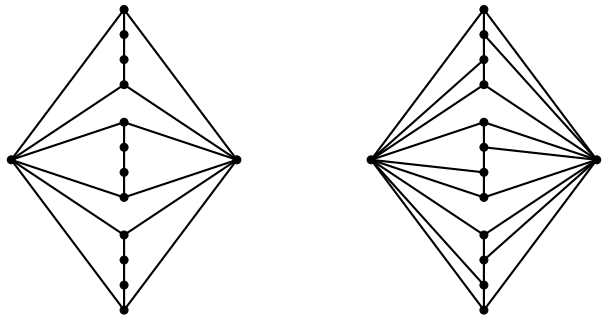
\begin{figure}[ht]
\centering
\begin{tikzpicture}
\draw[black, thick] (0,0.5) -- (-1.5,0);
\draw[black, thick] (0,0.5) -- (1.5,0);
\draw[black, thick] (0,-0.5) -- (-1.5,0);
\draw[black, thick] (0,-0.5) -- (1.5,0);
\draw[black, thick] (0,-0.5) -- (0,0.5);
\draw[black, thick] (0,2) -- (-1.5,0);
\draw[black, thick] (0,2) -- (1.5,0);
\draw[black, thick] (0,1) -- (-1.5,0);
\draw[black, thick] (0,1) -- (1.5,0);
\draw[black, thick] (0,1) -- (0,2);
\draw[black, thick] (0,-2) -- (-1.5,0);
\draw[black, thick] (0,-2) -- (1.5,0);
\draw[black, thick] (0,-1) -- (-1.5,0);
\draw[black, thick] (0,-1) -- (1.5,0);
\draw[black, thick] (0,-1) -- (0,-2);
\filldraw[black] (-1.5,0) circle (1.5pt)
node[anchor=east] {$l_1$};
\filldraw[black] (1.5,0) circle (1.5pt)
node[anchor=west] {$l_2$};
\filldraw[black] (0,0.5) circle (1.5pt)
node[anchor=west] {$w_2$};
\filldraw[black] (0,-0.5) circle (1.5pt);
\filldraw[black] (0,1/6) circle (1.5pt);
\filldraw[black] (0,-1/6) circle (1.5pt);
\filldraw[black] (0,2) circle (1.5pt)
node[anchor=west] {$w_1$};
\filldraw[black] (0,1) circle (1.5pt);
\filldraw[black] (0,5/3) circle (1.5pt);
\filldraw[black] (0,4/3) circle (1.5pt);
\filldraw[black] (0,-2) circle (1.5pt);
\filldraw[black] (0,-1) circle (1.5pt)
node[anchor=west] {$w_3$};
\filldraw[black] (0,-5/3) circle (1.5pt);
\filldraw[black] (0,-4/3) circle (1.5pt);
\end{tikzpicture}
\qquad \qquad
\begin{tikzpicture}
\draw[black, thick] (0,0.5) -- (-1.5,0);
\draw[black, thick] (0,0.5) -- (1.5,0);
\draw[black, thick] (0,-0.5) -- (-1.5,0);
\draw[black, thick] (0,-0.5) -- (1.5,0);
\draw[black, thick] (0,-0.5) -- (0,0.5);
\draw[black, thick] (0,1/6) -- (1.5,0);
\draw[black, thick] (0,-1/6) -- (-1.5,0);
\draw[black, thick] (0,2) -- (-1.5,0);
\draw[black, thick] (0,2) -- (1.5,0);
\draw[black, thick] (0,1) -- (-1.5,0);
\draw[black, thick] (0,1) -- (1.5,0);
\draw[black, thick] (0,1) -- (0,2);
\draw[black, thick] (0,5/3) -- (1.5,0);
\draw[black, thick] (0,4/3) -- (-1.5,0);
\draw[black, thick] (0,-2) -- (-1.5,0);
\draw[black, thick] (0,-2) -- (1.5,0);
\draw[black, thick] (0,-1) -- (-1.5,0);
\draw[black, thick] (0,-1) -- (1.5,0);
\draw[black, thick] (0,-1) -- (0,-2);
\draw[black, thick] (0,-5/3) -- (-1.5,0);
\draw[black, thick] (0,-4/3) -- (1.5,0);
\filldraw[black] (-1.5,0) circle (1.5pt);
\filldraw[black] (1.5,0) circle (1.5pt);
\filldraw[black] (0,0.5) circle (1.5pt)
node[anchor=west] {$w_2$};
\filldraw[black] (0,-0.5) circle (1.5pt)
node[anchor=west] {$w_2'$};
\filldraw[black] (0,1/6) circle (1.5pt);
\filldraw[black] (0,-1/6) circle (1.5pt);
\filldraw[black] (0,2) circle (1.5pt)
node[anchor=west] {$w_1$};
\filldraw[black] (0,1) circle (1.5pt)
node[anchor=west] {$w_1'$};
\filldraw[black] (0,5/3) circle (1.5pt);
\filldraw[black] (0,4/3) circle (1.5pt);
\filldraw[black] (0,-2) circle (1.5pt)
node[anchor=west] {$w_3'$};
\filldraw[black] (0,-1) circle (1.5pt)
node[anchor=west] {$w_3$};
\filldraw[black] (0,-5/3) circle (1.5pt);
\filldraw[black] (0,-4/3) circle (1.5pt);
\end{tikzpicture}
\caption{The graph $\Gamma_2$ and $\Gamma_3$.}
\label{gamma23}
\end{figure}
We give more examples satisfying the assumptions in Corollary~\ref{beingAHcor}. These examples suggest infinitely many cases of having acylindrically hyperbolic $\Out^*(A_{\Gamma})$.

Define a graph $\Gamma_3$ as illustrated to the right in Figure \ref{gamma23}. In this case, $\mathrm{PSO}(A_{\Gamma_3})$ is isomorphic to a subgroup of Fouxe-Rabinovitch group $\Out(\mathbb{F}_6;$ $\{\langle x_1,x_2 \rangle,$ $\langle x_3,x_4 \rangle, \langle x_5,x_6 \rangle\}^t)$ where $x_i's$ are the usual free generators of $\mathbb{F}_6$. See \cite{day2019relative} for the details.

The existence of a maximal SIL-pair system allows us to decompose the graph $\Gamma_3$. It has three simultaneously shared components $C_1$, $C_2$, and $C_3$. Each component $C_i$ has four vertices $w_i, w_i', v_i, v_i'$ where $w_i$ and $w_i'$ are the topmost and the bottommost vertices respectively, $v_i$ is the vertex adjacent to $w_i$, and $v_i'$ is the vertex adjacent to $w_i'$. Let $C^i=\{v_i', w_i'\}$ and $C'^i=\{v_i, w_i\}$. By Theorem \ref{semidirect}, $\mathrm{PSO}(A_{\Gamma_3})=H_3\rtimes K$ where
\[ H_3=\langle P_{w_1}^{C_3}, P_{w_1'}^{C_3}, P_{w_2}^{C_3}, P_{w_2'}^{C_2}, P_{w_3}^{C_2}, P_{w_3'}^{C_2}\rangle\]
and
\[K= \langle P_{w_1}^{C^1},  P_{w_1'}^{C'^{1}}, P_{w_2}^{C^2},  P_{w_2'}^{C'^{2}}, P_{w_3}^{C^3},  P_{w_3'}^{C'^{3}}\rangle.\]
Let $t=P_{w_1}^{C_3}$. Then $\mathrm{PSO}(A_{\Gamma_3})$ is an HNN-extension group of $\langle P_{w_1'}^{C_3}, P_{w_2}^{C_3}, P_{w_2'}^{C_2},$ $P_{w_3}^{C_2}, P_{w_3'}^{C_2},  P_{w_1}^{C^1},  P_{w_1'}^{C'^{1}}, P_{w_2}^{C^2},  P_{w_2'}^{C'^{2}}, P_{w_3}^{C^3},  P_{w_3'}^{C'^{3}}\rangle$ with associate subgroups
\[A=B=\langle P_{w_1'}^{C_3}, P_{w_1}^{C^1},  P_{w_1'}^{C'^{1}}, P_{w_2}^{C^2},  P_{w_2'}^{C'^{2}}, P_{w_3}^{C^3},  P_{w_3'}^{C'^{3}}\rangle.\] The relations containing $t$ can be rewritten as follows:
\begin{enumerate}
    \item $tP_{w_1}^{C^{1}} t^{-1}=P_{w_1}^{C^{1}}$,
    \item $tP_{w_1'}^{C_3} t^{-1}=(P_{w_1}^{C^{1}})^{-1} P_{w_1'}^{C_3} P_{w_1}^{C^{1}}$,
    \item $tP_{w_1'}^{C'^{1}}t^{-1}=P_{w_1'}^{C'^{1}} P_{w_1'}^{C_3} (P_{w_1}^{C^1})^{-1} (P_{w_1'}^{C_3})^{-1} P_{w_1}^{C^1}$,
    \item $tP_{w_i}^{C^i} t^{-1}=P_{w_i}^{C^i}$ for each $i=2,3$,
    \item $tP_{w_i'}^{C'^i} t^{-1}=P_{w_i'}^{C'^i}$ for each $i=2,3$.
\end{enumerate}
Since there are no transvections in $\Gamma_3$, $\mathrm{PSO}(A_{\Gamma_3})=\Out^*(A_{\Gamma_3})$ is a finite index subgroup of $\Out(A_{\Gamma_3})$, and by Corollary \ref{beingAHcor}, it is acylindrically hyperbolic. Each $C_i$ consists of two vertices that define exactly one subordinate partial conjugation.
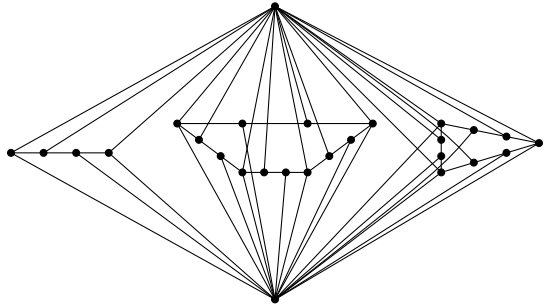
\begin{figure}[h]
\centering
\begin{tikzpicture}[scale = 1.3]
\filldraw[black] (0,1.5) circle (1pt);
\filldraw[black] (0,-1.5) circle (1pt);
\filldraw[black] (-1.7,0) circle (1pt);
\filldraw[black] (-61/30,0) circle (1pt);
\filldraw[black] (-71/30,0) circle (1pt);
\filldraw[black] (-2.7,0) circle (1pt);
\filldraw[black] (1,0.3) circle (1pt);
\filldraw[black] (-1,0.3) circle (1pt);
\filldraw[black] (1/3,-0.2) circle (1pt);
\filldraw[black] (-1/3,-0.2) circle (1pt);
\filldraw[black] (-1/9,-0.2) circle (1pt);
\filldraw[black] (1/9,-0.2) circle (1pt);
\filldraw[black] (5/9,-1/30) circle (1pt);
\filldraw[black] (-5/9,-1/30) circle (1pt);
\filldraw[black] (7/9,4/30) circle (1pt);
\filldraw[black] (-7/9,4/30) circle (1pt);
\filldraw[black] (1/3,0.3) circle (1pt);
\filldraw[black] (-1/3,0.3) circle (1pt);
\filldraw[black] (1.7,0.3) circle (1pt);
\filldraw[black] (1.7,-0.2) circle (1pt);
\filldraw[black] (1.7,4/30) circle (1pt);
\filldraw[black] (1.7,-1/30) circle (1pt);
\filldraw[black] (2.7,0.1) circle (1pt);
\filldraw[black] (71/30,1/6) circle (1pt);
\filldraw[black] (61/30,7/30) circle (1pt);
\filldraw[black] (71/30,0) circle (1pt);
\filldraw[black] (61/30,-0.1) circle (1pt);
\draw[black] (1.7,0.3) -- (1.7,-0.2);
\draw[black] (1.7,0.3) -- (2.7,0.1);
\draw[black] (1.7,-0.2) -- (2.7,0.1);
\draw[black] (1.7,4/30) -- (0,1.5);
\draw[black] (71/30,1/6) -- (0,1.5);
\draw[black] (61/30,-0.1) .. controls (1.7,0.25) .. (0,1.5);
\draw[black] (1.7,-1/30) -- (0,-1.5);
\draw[black] (61/30,7/30) .. controls (1.7,-0.16) .. (0,-1.5);
\draw[black] (71/30,0) -- (0,-1.5);
\draw[black] (1.7,0.3) -- (0,1.5);
\draw[black] (1.7,0.3) -- (0,-1.5);
\draw[black] (1.7,-0.2) -- (0,1.5);
\draw[black] (1.7,-0.2) -- (0,-1.5);
\draw[black] (2.7,0.1) -- (0,1.5);
\draw[black] (2.7,0.1) -- (0,-1.5);
\draw[black] (1,0.3) -- (0,1.5);
\draw[black] (-1,0.3) -- (0,1.5);
\draw[black] (1/3,-0.2) -- (0,1.5);
\draw[black] (-1/3,-0.2) -- (0,1.5);
\draw[black] (-7/9,4/30) -- (0,1.5);
\draw[black] (-1/9,-0.2) -- (0,1.5);
\draw[black] (5/9,-1/30) -- (0,1.5);
\draw[black] (1/3,0.3) -- (0,1.5);
\draw[black] (-1/3,0.3) -- (0,-1.5);
\draw[black] (7/9,4/30) -- (0,-1.5);
\draw[black] (-5/9,-1/30) -- (0,-1.5);
\draw[black] (1/9,-0.2) -- (0,-1.5);
\draw[black] (1,0.3) -- (0,-1.5);
\draw[black] (-1,0.3) -- (0,-1.5);
\draw[black] (1/3,-0.2) -- (0,-1.5);
\draw[black] (-1/3,-0.2) -- (0,-1.5);
\draw[black] (1,0.3) -- (-1,0.3);
\draw[black] (1,0.3) -- (1/3,-0.2);
\draw[black] (-1,0.3) -- (-1/3,-0.2);
\draw[black] (-1/3,-0.2) -- (1/3,-0.2);
\draw[black] (-1.7,0) -- (0,1.5);
\draw[black] (-71/30,0) -- (0,1.5);
\draw[black] (-2.7,0) -- (0,1.5);
\draw[black] (-1.7,0) -- (0,-1.5);
\draw[black] (-61/30,0) -- (0,-1.5);
\draw[black] (-2.7,0) -- (0,-1.5);
\draw[black] (-1.7,0) -- (-2.7,0);
\end{tikzpicture}
\caption{The graph $\Gamma(2,4,3)$.}
\label{Fig:Gamma(2,4,3)}
\end{figure}

For each triple $(p,q,r)$ of integers with $p,q,r\geq 2$, we defined the graph $\Gamma(p,q,r)$ above Corollary \ref{maincor2}. See Figure \ref{Fig:Gamma(2,4,3)} as an example. Indeed, $\Gamma_3$ in Figure \ref{gamma23} is the same as $\Gamma(2,2,2)$. One can construct a maximal SIL-pair system in $\Gamma(p,q,r)$ with three simultaneous shared components such that each component has $p$, $q$, or $r$ vertices defining exactly one subordinate partial conjugation. The other vertices define only trivial partial conjugations in $\Out(A_{\Gamma(p,q,r)})$, and there are no transvections in $\Gamma(p,q,r)$. Consequently, $\mathrm{PSO}(A_{\Gamma(p,q,r)})=\Out^*(A_{\Gamma(p,q,r)})$ is an acylindrically hyperbolic finite index subgroup of $\Out(A_{\Gamma(p,q,r)})$ by Corollary~\ref{beingAHcor}.

\section{Proof of the main corollary}\label{sec:proof-main-cor}

In this section we prove Corollary~\ref{maincor1}. The argument proceeds in two steps. We first establish that $\Gamma$ has no SIL-pair with high probability under the stated hypotheses (Lemma~\ref{lem:no-sil-random} below); we then deduce the failure of acylindrical hyperbolicity by analyzing equivalent and comparable pairs in $\Gamma$ and applying Theorem~\ref{mainthmA}.

\begin{lem}\label{lem:no-sil-random}
Let $\Gamma=\Gamma(n,p)$ and let $q=1-p$.
Assume
\[
p \ \ge\ \frac{\log n+2\log\log n+\omega(n)}{n}\quad\text{with}\quad \omega(n)\to\infty,
\qquad\text{and}\qquad qn\to\infty.
\]
Then $\Gamma$ has no SIL-pair with high probability.
\end{lem}

\begin{proof}
We split into two cases according to $p$. It suffices to show that $\Gamma$ has no SIL-pair with high probability separately in the cases $p\le \tfrac12$ and $p>\tfrac12$.

\smallskip\noindent
\textbf{Case 1: $p\le \tfrac12$.}
We further split according to $pn$.

\smallskip\noindent
\emph{Subcase 1a: $pn\le 10\log n$.}
Then $p\le 10\log n/n$, hence $n^5p^6\to 0$ and
\[
\mathbb P\Big(\exists\, a\neq b:\ |\lk(a)\cap\lk(b)|\ge 3\Big)
\ \le\ \binom{n}{2}\binom{n}{3}p^6
\ =\ O(n^5p^6)\ =\ o(1).
\]
Therefore, with high probability,
\begin{equation}\label{eq:max-common-nei-2}
|\lk(a)\cap\lk(b)|\le 2\quad\text{for all }a\neq b.
\end{equation}
We claim that, under our assumption
\[
p\ge \frac{\log n+2\log\log n+\omega(n)}{n}\qquad\text{with }\omega(n)\to\infty,
\]
the random graph $\Gamma=\Gamma(n,p)$ is $3$--vertex-connected with high probability.

\smallskip\noindent
\emph{Minimum degree.}
We first show that $\delta(\Gamma)\ge 3$ with high probability, where $\delta(\Gamma)=\min_{v\in \Gamma}\deg_\Gamma(v)$ denotes the minimum degree of $\Gamma$.
For $k\in\{0,1,2\}$ let $X_k$ denote the number of vertices of degree $k$ in $\Gamma$.
Since $X_k\ge 0$ is integer-valued, Markov's inequality gives
\[
\mathbb P(X_k\ge 1)\le \mathbb E[X_k].
\]
For a fixed vertex $v$ we have $\deg(v)\sim \mathrm{Bin}(n-1,p)$, hence
\[
\mathbb E[X_k]
= n\,\mathbb P(\deg(v)=k)
= n\binom{n-1}{k}p^k(1-p)^{n-1-k}.
\]
Using $(1-p)^{n-1-k}\le e^{-p(n-1-k)}$ and our assumption $pn\ge \log n+2\log\log n+\omega(n)$, we obtain
\[
\mathbb E[X_0]\le n e^{-p(n-1)} = o(1),\qquad
\mathbb E[X_1]\le n(n-1)p\,e^{-p(n-2)} = o(1),
\]
and
\[
\mathbb E[X_2]\le n\binom{n-1}{2}p^2\,e^{-p(n-3)} = o(1).
\]
Therefore,
\[
\mathbb P(\delta(\Gamma)\le 2)\le \sum_{k=0}^2 \mathbb P(X_k\ge 1)
\le \sum_{k=0}^2 \mathbb E[X_k]=o(1),
\]
so $\delta(\Gamma)\ge 3$ with high probability.

\smallskip\noindent
\emph{From $\delta\ge 3$ to $3$--vertex-connected.}
Let $K_n$ denote the complete graph on $n$ vertices and set $N=\binom{n}{2}=|E(K_n)|$.
Let $(G_t)_{t=0}^N$ be the random graph process obtained by ordering the $N$ edges of $K_n$ uniformly at random and letting $G_t$ consist of the first $t$ edges.
For a graph property $Q$, write
\[
\tau(Q)=\min\{t:G_t \text{ has } Q\}.
\]
Couple $\Gamma=\Gamma(n,p)$ with this process as follows: let $M\sim\mathrm{Bin}(N,p)$ and set $\Gamma:=G_M$.
Then $\Gamma$ has the same distribution as $\Gamma(n,p)$.

By \cite[Chapter~7.2, Theorem~7.4]{bollobas2001random}, asymptotically almost surely
\[
\tau(\delta\ge 3)=\tau(\kappa\ge 3),
\]
Here $\kappa(\Gamma)$ denotes the vertex connectivity of $\Gamma$, i.e.\ the largest integer $k$ such that removing any set of at most $k-1$ vertices keeps $\Gamma$ connected.
On this event, if $\delta(\Gamma)\ge 3$ then $\tau(\delta\ge 3)\le M$, hence $\tau(\kappa\ge 3)\le M$, so $\kappa(\Gamma)\ge 3$.
In particular, removing at most two vertices cannot disconnect $\Gamma$.

Thus, combining with \eqref{eq:max-common-nei-2}, for every nonadjacent $a,b$ the graph $\Gamma-(\lk(a)\cap\lk(b))$ is connected with high probability, and therefore no SIL-pair exists with high probability.

\smallskip\noindent
\emph{Subcase 1b: $pn>10\log n$.}
Fix an ordered pair $(a,b)$ of distinct vertices and let
\[
S=\lk(a)\cap\lk(b),\qquad R=V(\Gamma)\setminus(\{a,b\}\cup S).
\]
Define
\[
A=\{u\in R:\{a,u\}\in E(\Gamma),\{b,u\}\notin E(\Gamma)\},\,
B=\{u\in R:\{b,u\}\in E(\Gamma),\{a,u\}\notin E(\Gamma)\}.
\]
\emph{Claim.} If $A\cup B\neq\varnothing$ and $\Gamma[R]$ is connected, where $\Gamma[R]$ be the induced subgraph on $R$, then $(a,b)$ is not a SIL-pair.

Indeed, if $A\neq\varnothing$ then some vertex of $R$ is adjacent to $a$ but not to $b$, hence lies in the same connected component as $a$ inside $\Gamma-S$; if $\Gamma[R]$ is connected, then all of $R$ lies in that component, so $\Gamma-S$ has no component disjoint from $\{a,b\}$.
The same holds if $B\neq\varnothing$.

Hence,
\[
\mathbb P\big((a,b)\text{ is a SIL-pair}\big)
\ \le\ \mathbb P(A\cup B=\varnothing)\ +\ \mathbb P(\Gamma[R]\text{ disconnected}).
\]

First, for $n$ large enough, we have
\[
\mathbb P(A\cup B=\varnothing)=(p^2+q^2)^{n-2}=(1-2pq)^{n-2}\le e^{-2pq(n-2)}
\le e^{-p(n-2)}\le e^{-pn/2} \le n^{-5},
\]
using $q\ge 1/2$ and $pn>10\log n$.

Next, $|S|\sim \mathrm{Bin}(n-2,p^2)$ with mean $\mu=(n-2)p^2$.
We claim that $\mathbb P(|S|\ge pn)=o(n^{-2})$.
Indeed, we distinguish two cases.

\smallskip\noindent
If $p\ge \tfrac14$, then we have $pn\ge 2\mu$ for $n$ large enough.
By a Chernoff bound,
\[
\mathbb P(|S|\ge pn)\ \le\ \mathbb P(|S|\ge 2\mu)\ \le\ \exp(-c\mu)\ \le\ \exp(-c'n)
\]
for some constants $c,c'>0$, hence $\mathbb P(|S|\ge pn)\le n^{-10}$ for all sufficiently large $n$.

\smallskip\noindent
If $p<\tfrac14$, let $k:=\lceil pn\rceil$ (so $k\ge pn>10\log n$).
Since $|S|\sim \mathrm{Bin}(n-2,p^2)$, the union bound gives
\[
\mathbb P(|S|\ge pn) = \mathbb P(|S|\ge k)
\le \binom{n}{k}(p^2)^k.
\]
Using $\binom{n}{k}\le \left(\frac{en}{k}\right)^k$ and $k\ge pn$, we get
\[
\mathbb P(|S|\ge pn)
\le \left(\frac{en}{k}\right)^k p^{2k}
\le \left(\frac{en}{pn}\right)^k p^{2k}
= (ep)^k.
\]
Since $p<\tfrac14$, we have $ep\le e/4<1$, hence
\[
(ep)^k\ \le\ (e/4)^k\ \le\ (e/4)^{10\log n}
= n^{-10\log(4/e)}=n^{-3.86...}.
\]
In particular, $\mathbb P(|S|\ge pn)=o(n^{-2})$.

On the event $|S|<pn$, we have
\[
|R|=n-2-|S|\ \ge\ n-2-pn\ \ge\ n/2
\qquad(\text{since }p\le 1/2).
\]
Let $m=|R|$. Conditional on $m\ge n/2$ and on the edges incident to $a$ or $b$, the induced graph $\Gamma[R]$ is distributed as $\Gamma(m,p)$.

We bound $\mathbb P(\Gamma(m,p)\text{ disconnected})$ by a cut union bound.
If $\Gamma(m,p)$ is disconnected, then there exists a nonempty $T\subseteq R$ with $|T|\le \lfloor m/2 \rfloor$ such that there are no edges between $T$ and $R\setminus T$.
Thus, using $q=1-p$ and $q\le e^{-p}$,
\[
\mathbb P(\Gamma(m,p)\text{ disconnected})
\le \sum_{k=1}^{\lfloor m/2\rfloor} \binom{m}{k} q^{k(m-k)}
\le \sum_{k=1}^{\lfloor m/2\rfloor} \binom{m}{k} e^{-p k(m-k)}.
\]
For $k=1$, this is at most $m e^{-p(m-1)}\le n e^{p-5\log n}\le e^{1/2}n^{-4}$ since $m\ge n/2$ and $pn>10\log n$.
For $2\le k\le m/2$, we have $m-k\ge m/2$, hence
\[
\binom{m}{k}e^{-p k(m-k)}\le \left(\frac{em}{k}\right)^k e^{-pkm/2},
\]
and
\[
\sum_{k\ge2}\left(\frac{em}{k}\right)^k e^{-pkm/2}
\le \sum_{k\ge2} (en\cdot n^{-2.5})^k
=\frac{(en^{-1.5})^2}{1-en^{-1.5}}
=O(n^{-3}).
\]
Thus, $\mathbb P(\Gamma(m,p)\text{ disconnected})=O(n^{-3})$. Therefore, for each ordered pair $(a,b)$ we have
\[
\mathbb P\big((a,b)\text{ is a SIL-pair}\big)=O(n^{-3}).
\]

A union bound over at most $n(n-1)$ ordered pairs $(a,b)$ shows that with high probability no SIL-pair exists.

\smallskip\noindent
\textbf{Case 2: $p>\tfrac12$.}
Fix an ordered triple of distinct vertices $(a,b,u)$.
Let $E_{a,b,u}$ be the event that
\begin{enumerate}
\item $\{a,b\}\notin E(\Gamma)$;
\item $\{u,a\}\notin E(\Gamma)$ and $\{u,b\}\notin E(\Gamma)$;
\item for every $x\in \Gamma\setminus\{a,b,u\}$, if $x$ is adjacent to exactly one of $a,b$, then $\{u,x\}\notin E(\Gamma)$.
\end{enumerate}
We claim that if $(a,b)$ is a SIL-pair, then $E_{a,b,u}$ holds for some $u$. Indeed, if $C$ is a connected component of $\Gamma-(\lk(a)\cap\lk(b))$ containing neither $a$ nor $b$, pick $u\in C$. Then $u$ is adjacent to neither $a$ nor $b$, and if $u$ had an edge to some vertex $x$ adjacent to exactly one of $a,b$, then $u$ would be connected to $a$ or $b$ in $\Gamma-(\lk(a)\cap\lk(b))$, contradicting $u\in C$.

Now we bound $\mathbb P(E_{a,b,u})$.
The conditions in (1)--(2) contribute a factor $q^3$.
For each $x\notin\{a,b,u\}$, the condition in (3) fails precisely if $x$ is adjacent to exactly one of $a,b$ (probability $2pq$) and $\{u,x\}\in E(\Gamma)$ (probability $p$), so the failure probability is $2p^2q$.
Hence,
\[
\mathbb P(E_{a,b,u}) \ \le\ q^3(1-2p^2q)^{n-3}\ \le\ q^3\exp\!\big(-2p^2q(n-3)\big).
\]
Since $p>\tfrac12$, we have $p^2\ge 1/4$, so
\[
\mathbb P(E_{a,b,u})\ \le\ q^3\exp\!\big(-\tfrac12\,qn\big).
\]
By the union bound over all ordered triples $(a,b,u)$, the probability that some $E_{a,b,u}$ occurs is at most
\[
n(n-1)(n-2)\,q^3\exp\!\big(-\tfrac12\,qn\big)\ \le\ (qn)^3\,\exp\!\big(-\tfrac12\,qn\big)\ \xrightarrow[n\to\infty]{}\ 0,
\]
since $qn\to\infty$ by assumption. By the claim, this implies that there is no SIL-pair with high probability.

Combining Cases 1 and 2 completes the proof.
\end{proof}

\begin{proof}[Proof of Corollary~\ref{maincor1}]
By Lemma~\ref{lem:no-sil-random}, $\Gamma$ has no SIL-pair with high probability.
Work on this event.

By Theorem~\ref{mainthmA}, since $\Gamma$ has no SIL-pair, if $\Out(A_\Gamma)$ were acylindrically hyperbolic then there would exist an equivalence class $\{a,b\}$ (with respect to the link-star equivalence) such that the only ordered comparable pairs $(v,w)$ with $v\le w$ are $(a,b)$ and $(b,a)$.
We show that this necessary uniqueness fails with high probability, splitting into two cases. Write $t=npq$.

\smallskip\noindent
\textbf{Case 1: $t\ge \log n+\log\log n$.}
For distinct vertices $a,b$, we have
\[
\mathbb P(a\sim b)=(p^2+q^2)^{n-2}=(1-2pq)^{n-2}\le \exp\!\big(-2pq(n-2)\big)
= \exp\!\big(-2t(1-2/n)\big).
\]
Hence the expected number of equivalent pairs is at most
\[
\binom{n}{2}\,e^{-2t(1-2/n)}\ \le\ \frac{n^2}{2}\,e^{-2(\log n+\log\log n)(1-2/n)}\ =\ o(1).
\]
By Markov's inequality, with high probability, there is no equivalent pair at all, so the necessary condition from Theorem~\ref{mainthmA} fails. Thus, $\Out(A_\Gamma)$ is not acylindrically hyperbolic.

\smallskip\noindent
\textbf{Case 2: $t<\log n+\log\log n$.}
Note that $pq=t/n=o(1)$. Let $X=\#\{(w,v):w\neq v,\ w\le v\}$ be the number of ordered comparable pairs.

For fixed $w\neq v$, one has
\[
\mathbb P(w\le v)=(1-pq)^{n-2},
\]
since for each $u\neq w,v$ the forbidden pattern ``$\{u,w\}\in E(\Gamma)$ and $\{u,v\}\notin E(\Gamma)$'' has probability $pq$, independently across $u$.
Therefore
\[
\mathbb E[X]=n(n-1)(1-pq)^{n-2}.
\]
Moreover,
\[
(1-pq)^{n-2}=\exp\!\big((n-2)\log(1-t/n)\big)
\ge \exp\!\big(-(n-2)\tfrac{t/n}{1-t/n}\big)
= \exp\!\Big(-\tfrac{t(1-2/n)}{1-t/n}\Big).
\]
Since $t/n\to 0$, the exponent satisfies $\tfrac{t(1-2/n)}{1-t/n}=t+o(1)$, and therefore
\[
(1-pq)^{n-2}\ge e^{-t-o(1)}.
\]
Consequently,
\[
\mathbb E[X]\ge n(n-1)e^{-t-o(1)}.
\]
If $t<\log n+\log\log n$, then $e^{-t}> \frac{1}{n\log n}$, and hence
\[
\mathbb E[X]\ge \frac{n(n-1)}{n\log n}\,e^{-o(1)}\ \longrightarrow\ \infty.
\]

We now show $\mathrm{Var}(X)=o(\mathbb E[X]^2)$.
Write $X=\sum_{w\neq v} I_{w,v}$, where $I_{w,v}$ is the indicator of $w\le v$, i.e. $I_{w,v}=1$ if $w\le v$ and $I_{w,v}=0$ otherwise.
If $(w,v)$ and $(x,y)$ are ordered pairs with four distinct vertices, then a direct computation gives
\[
\mathbb E[I_{w,v}I_{x,y}]
=(p^4+pq+q^4)\,(1-pq)^{2(n-4)}.
\]
Since $pq=o(1)$, one has
\[
\frac{p^4+pq+q^4}{(1-pq)^4}=1+O(pq)=1+o(1),
\]
and hence, the total contribution of disjoint ordered pairs to $\mathbb E[X^2]$ is $(1+o(1))\mathbb E[X]^2$.
The remaining overlap cases involve at most three distinct vertices, hence there are $O(n^3)$ such ordered pairs, each contributing at most $(1-pq)^{2(n-3)}$; therefore, their total contribution is
\[
O\bigl(n^3(1-pq)^{2(n-3)}\bigr)=o\bigl(n^4(1-pq)^{2(n-2)}\bigr)=o(\mathbb E[X]^2).
\]
Consequently, $\mathrm{Var}(X)=o(\mathbb E[X]^2)$. By Chebyshev's inequality, for every $\varepsilon>0$,
\[
\mathbb P\!\left(\left|\frac{X}{\mathbb E[X]}-1\right|\ge \varepsilon\right)
=\mathbb P\!\left(|X-\mathbb E[X]|\ge \varepsilon\,\mathbb E[X]\right)
\le \frac{\mathrm{Var}(X)}{\varepsilon^2\,\mathbb E[X]^2}.
\]
Since $\mathrm{Var}(X)=o(\mathbb E[X]^2)$, the right-hand side tends to $0$, hence $X/\mathbb E[X]\to 1$ in probability.
In particular, $X\to\infty$ with high probability.

Thus, with high probability, there are many ordered relations $w\le v$, so it is impossible that the only ordered comparable pairs are $(a,b)$ and $(b,a)$ for a single equivalence class $\{a,b\}$.
Hence, the necessary condition from Theorem~\ref{mainthmA} fails, and $\Out(A_\Gamma)$ is not acylindrically hyperbolic.

Combining the two cases completes the proof.
\end{proof}
\medskip
\bibliographystyle{alpha}
\bibliography{AH}
\end{document}